\documentclass[12pt]{amsart}
\usepackage{amssymb,amsmath,amsfonts,latexsym}
\usepackage{hyperref}
\newtheorem{theorem}{Theorem}[section]

\newtheorem{proposition}[theorem]{Proposition}
\newtheorem{corollary}[theorem]{Corollary}
\theoremstyle{definition}

\newtheorem{remark}[theorem]{Remark}
\numberwithin{equation}{section}

\title[Rational orbits in some PV's associated to $Sp_{6}$ revisited]{Rational orbits in some prehomogeneous vector spaces associated to $Sp_{6}$ revisited}

\author{Sayan Pal}

\address{Indian Statistical Institute, Statistics and Mathematics Unit, 8th Mile, Mysore Road, Bangalore, 560059, India}

\email{syn.pal98@gmail.com}

\begin{document}



\begin{abstract}\noindent Let $k$ be a field with $\text{char}(k)\neq 2$. We prove that all maximal flags of composition algebras over $k$, appear as the $k$-rational $Sp_{6}$-orbits in a Zariski-dense $Sp_{6}$-invariant subset $V^{ss}\subset V=\wedge^{3}V_{6}$, where $V_{6}$ is the standard $6$-dimensional irreducible representation of $Sp_{6}$. This gives an arithmetic interpretation for the orbit spaces of the semi-stable sets in the prehomogeneous vector spaces $(Sp_{6}\times GL_{1}^{2},V)$ and $(GSp_{6}\times GL_{1}^{2},V)$. We also get all reduced Freudenthal algebras of dimensions $6$ and $9$, represented by the same orbit spaces.

\end{abstract}

\maketitle

\tableofcontents

\section{Introduction} \label{1} The main aim of this paper is to describe all maximal flags of composition algebras over a field $k$, in terms of the orbit spaces of the semi-stable sets in some prehomogeneous vector spaces. We prove that each orbit from the orbit space of a particular Zariski-dense subset inside a $Sp_{6}$-representation represents a unique $4$-flag of the form $(k\subset K\subset Q\subset O)$; where $K, Q$ and $O$ are quadratic, quaternion, and octonion algebras, respectively. And we can get all such $4$-flags from the same orbit space (see Section \ref{4}).

\vspace{2 mm}

Let $G$ be a connected reductive algebraic group and $V$ be a rational representation of $G$, defined over the field $k$. We assume that $(G,V)$ is a prehomogeneous vector space (called PV, in short) with a relative invariant $f\in k[V]$. We define $V^{ss}=\{v\in V:f(v)\neq 0\}$. Then, the $k$-rational orbit space $V^{ss}_{k}/G_{k}$ has a phenomenon of classifying several interesting algebraic structures (see \cite{WY}, \cite{KY}, \cite{PS}). But in many cases the descriptions of these algebraic objects are not known. In this paper, we will give an answer to this question for the PV's $(Sp_{6}\times GL_{1}^{2},\wedge^{3}V_{6})$ and $(GSp_{6}\times GL_{1}^{2},\wedge^{3}V_{6})$ (see Section \ref{5}).

\vspace{2 mm}

Let us consider the group $Sp_{6}$ and its standard irreducible representation $V_{6}$ of dimension $6$. Then $\wedge^{3}V_{6}$ is a representation of $Sp_{6}$ which decomposes into two irreducible components $X$ and $V_{6}$ itself, where $\text{dim}(X)=14$. Now, let $C$ be any split composition algebra over $k$ or the field $k$ itself, and $\mathcal{H}_{3}(C)=\mathcal{H}_{3}(C,I_{3})$ be the corresponding reduced Freudenthal algebra. Then we can define an algebra structure on the following space \[Z(C)=\left(\begin{array}{cc}
k & \mathcal{H}_{3}(C) \\
\mathcal{H}_{3}(C) & k
\end{array}\right),\] with a homogeneous polynomial of degree $4$ defined on it (see \cite{AF}, \cite{BM}). These are structurable algebras with skew-dimension $1$, and they are very interesting as representations of some subgroups of $GL(Z(C))$, which leave the degree $4$ polynomial defined on $Z(C)$ invariant (see \cite{LM}). The cases where $\text{dim}(C)=1,4,8$, have been discussed in (\cite{PS}), (\cite{IJ}) and (\cite{HJ}), respectively. In particular, for $\text{dim}(C)=1$ (i.e., for $C=k$) we get $Z(C)$ as the irreducible $Sp_{6}$-representation $X\subset \wedge^{3}V_{6}$ of dimension $14$ (see \cite{PS} for details). In this paper, we will discuss the case where $\text{dim}(C)=2$, i.e., $C\simeq k\times k$. To avoid notational ambiguity, we denote $Z(k\times k)$ simply by $Z$. We have $\mathcal{H}_{3}(k\times k)\simeq M_{3}(k)^{+}$, the Jordan algebra arising from the associative algebra $M_{3}(k)$ (see \cite{KMRT}, Chapter IX). We will see that we can identify the $Sp_{6}$-representation $\wedge^{3}V_{6}$ as a vector space with the structurable algebra $Z$, and the quartic homogeneous polynomial defined on $Z$ gives us a $Sp_{6}$-invariant polynomial defined on $\wedge^{3}V_{6}$ (see Section \ref{3}).

\vspace{2 mm}

Again, in (\cite{PS}) we saw that the irreducible $Sp_{6}$-representation $X\subset \wedge^{3}V_{6}$ can be obtained as the span of the $Sp_{6}$-orbit of a $SL_{3}$-invariant alternating trilinear form defined on the following off-diagonal subspace of the split octonion algebra $Zorn(k)$, where $SL_{3}\subset \text{Aut}(Zorn(k))$ acts by fixing the diagonal entries. \[V_{6}\simeq \left(\begin{array}{cc}
0 & k^{3} \\
k^{3} & 0
\end{array}\right)\subset  Zorn(k)\] The octonion $Zorn(k)$ contains all composition algebras of dimensions $1,2$ and $4$ as subalgebras, and the $k$-rational orbit space in the resultant representation $X$ parametrizes all composition algebras defined over the field $k$ (see \cite{PS} for details). Here we will see that we can get a similar description of these algebras if we consider the entire $Sp_{6}$-representation $\wedge^{3}V_{6}=X\oplus V_{6}$, and this time the underlying orbit space can also distinguish all possible inclusions of one composition algebra into the other algebras as a subalgebra (see Section \ref{4}). In other words, this gives us all possible Cayley-Dickson doublings.

\vspace{2 mm}

We can see for a prehomogeneous vector space $(G,V)$ that the semi-stable set $V^{ss}$ is a single orbit over the algebraic closure $\overline{k}$ (see Section \ref{5}). But in general, the $k$-rational orbit space $V^{ss}_{k}/G_{k}$ consists of more than one orbit. This is one reason behind the fact that $V^{ss}_{k}/G_{k}$ classifies several algebraic structures $B$ over $k$ such that $B\otimes \overline{k}\simeq A$ for some fixed algebraic structure $A$ defined over $\overline{k}$, i.e., the $k$-rational orbit space in the semi-stable set classifies the $\overline{k}/k$-forms of some algebraic structure $A$ over $\overline{k}$ (see \cite{SJP}, Chapter III, Section $1$). In this case also, if we take the flag of composition algebras \[ \overline{k}\subset \overline{k}\times \overline{k}\subset M_{2}(\overline{k})\subset Zorn(\overline{k}),\] then it is the only maximal flag of composition algebras over $\overline{k}$ and we get all $\overline{k}/k$-forms of this flag from the orbit spaces in the semi-stable sets of the PV's $(Sp_{6}\times GL_{1}^{2}, \wedge^{3}V_{6})$ and $(GSp_{6}\times GL_{1}^{2},\wedge^{3}V_{6})$ (see Section \ref{5}). The same argument applies for the reduced Freudenthal algebras as well (see Section \ref{6}).

\vspace{2 mm}

In Section \ref{4}, we see that the group $Sp_{6}(k)$ contains any group of type $A_{1}$ and the groups of type $A_{2}$ of the form $SU(h)$ where $h$ is a trivial discriminant $K/k$-hermitian form of rank $3$ for some quadratic extension $K/k$. The groups of type $A_{1}$ appear as the group of automorphisms of an octonion algebra that fix some particular quaternion subalgebra inside the octonion. And the groups of type $A_{2}$ in the above form appear as the group of automorphisms of an octonion algebra fixing some particular quadratic subalgebra (see \cite{JN}, \cite{TML} for details). This gives us an idea of why these maximal flags of composition algebras appear in a very natural way from the $Sp_{6}$-representation.

\vspace{2 mm}

We start with some preliminaries in Section \ref{2}. In Section \ref{3}, we describe the symplectic group $Sp_{6}$, the representation $\wedge^{3}V_{6}$, and the invariants defined on the underlying vector space. In Section \ref{4}, we have described all maximal flags of composition algebras in terms of the $Sp_{6}(k)$-orbits in a Zariski-dense subset of $\wedge^{3}V_{6}$. In Section \ref{5}, we have computed the orbit spaces in the semi-stable sets for the PV's $(Sp_{6}\times GL_{1}^{2},\wedge^{3}V_{6})$ and $(GSp_{6}\times GL_{1}^{2}, \wedge^{3}V_{6})$. Finally, in Section \ref{6}, we have discussed what we can get for Freudenthal algebras using the results proved in the previous sections.

\vspace{2 mm}

Throughout the paper, we will always assume that the characteristic of the underlying field $k$ is different from $2$ unless mentioned otherwise. If $U$ is any object defined over the field $k$, by $U_{k}$ or $U(k)$ we will denote the $k$-rational points in that object. If $A, B(\subset A)$ are two algebras, by $\text{Aut}(A/B)$ we denote the group of automorphisms of $A$ fixing $B$ elementwise, and by $\text{Aut}(A,B)$ we mean the group of all automorphisms of $A$ that maps $B$ onto itself.

\section{Preliminaries}\label{2}

In this section, we define the representations called prehomogeneous vector spaces. We also describe some basic properties of the composition algebras and Freudenthal algebras. For more details on these topics, we refer to (\cite{SK}), (\cite{SV}), (\cite{KMRT}) and (\cite{GPR}).

\subsection{Prehomogeneous vector spaces} Let $G$ be a connected reductive algebraic group and $V$ be a finite dimensional rational representation of $G$, both defined over the field $k$. Then we say that $(G,V)$ is a \textit{prehomogeneous vector space} (PV) if the following two conditions are satisfied:

\begin{enumerate}
    \item $\exists $ a homogeneous polynomial $f\in k[V]$ and a group homomorphism $\chi:G\rightarrow GL_{1}$ such that $f(g.x)=\chi(g)f(x), \forall g\in G, x\in V$;
    \item $V^{ss}=\{x\in V:f(x)\neq 0\}$ is a single $G$-orbit over $\overline{k}$.
\end{enumerate}

In (\cite{SK}), this is the definition of regular PV's with reductive groups (see \cite{SK}, Remark $26$, page 73). However, we will consider this definition for a PV, following the notion of (\cite{IJ1}). We call $V^{ss}$ the set of \textit{semi-stable points}, which is a Zariski-dense $G$-orbit over $\overline{k}$. The polynomial $f$ in condition $(1)$ is called the \textit{relative invariant}.

\vspace{1 mm}

Let $G\subset GL_{n}$ be a connected reductive algebraic group for some $n\in \mathbb{N}$ and $r+s=n$ for some $r,s\in \mathbb{N}$. Take $G_{1}=G\times GL_{r}$, $G_{2}=G\times GL_{s}$, $V_{1}=V_{n}\otimes V_{r}$ and $V_{2}=V_{n}^{\ast}\otimes V_{s}$, where $V_{n},V_{r},V_{s}$ are vector spaces of dimensions $n,r,s$, respectively, and $V_{n}^{\ast}$ denotes the dual of $V_{n}$. Then $V_{1}$ and $V_{2}$ are representations of the groups $G_{1}$ and $G_{2}$, respectively, with obvious actions; and we have the following result from (\cite{IJ1}, Proposition $3.1$) and (\cite{SK}, Proposition $9$, page 38; Remark $26$, page $73$).

\begin{proposition}
    $(G_{1},V_{1})$ is a PV if and only if $(G_{2},V_{2})$ is a PV. Moreover, the orbit spaces in the semi-stable sets are in bijection, i.e., $V_{1}^{ss}(k)/G_{1}(k)\longleftrightarrow V_{2}^{ss}(k)/G_{2}(k)$.
\end{proposition}

The PV's $(G_{1},V_{1})$ and $(G_{2},V_{2})$ in the above form are called \textit{castling transforms} of each other. If $f$ and $g$ are the relative invariants, respectively, then $\text{deg}(f)=rd$ and $\text{deg}(g)=sd$, for some $d\in \mathbb{N}$ (see \cite{SK}, Proposition $18$, page 68). In general, if a PV $(G_{1},V_{1})$ can be obtained from another PV $(G_{2},V_{2})$ through finitely many castling transforms lying between them, we say that $(G_{1},V_{1})$ and $(G_{2},V_{2})$ are \textit{castling equivalent}. The above result says that if we can determine the orbit space in the semi-stable set $V^{ss}_{k}/G_{k}$ for a PV $(G,V)$, the same is determined for all PV's castling equivalent to $(G,V)$. For more details, we refer to (\cite{SK}), (\cite{IJ1}).

\subsection{Maximal flags of composition algebras} Let $k$ be a field with $\text{char}(k)\neq 2$. A \textit{composition algebra} $C$ over $k$ is a unital (not necessarily associative) algebra with a non-degenerate quadratic form $N_{C}$ defined on it such that \[N_{C}(xy) =N_{C}(x)N_{C}(y),\forall x,y\in C. \] The quadratic form $N_{C}$ is called the \textit{norm} of the composition algebra $C$ and we write the algebra as $(C,N_{C})$. Let $1_{C}\in C$ be the identity element and $b_{N_{C}}$ be the polar form of the norm $N_{C}$, i.e., \[b_{N_{C}}(x,y)=N_{C}(x+y)-N_{C}(x)-N_{C}(y), \forall x,y\in C.\] Then every element $x\in C$ satisfies the following quadratic equation \[x^{2}-b_{N_{C}}(x,1_{C}).x+N_{C}(x).1_{C}=0.\] We define a map on $C$ by $x\mapsto \overline{x}=b_{N_{C}}(x,1_{C}).1_{C}-x$, $x\in C$; which is called the \textit{conjugation} involution. From the above equation, we can easily see that $x\overline{x}=N_{C}(x),\forall x\in C$. So, if the norm $N_{C}$ is anisotropic, i.e., $N_{C}(x)\neq 0,\forall x\in C-\{0\}$; $C$ is a \textit{division} algebra. Otherwise, $C$ contains zero divisors, and these algebras are called \textit{split} composition algebras. The field $C=k$ itself is a division composition algebra with the norm $N_{C}(x)=x^{2}, x\in k$.

\vspace{2 mm}

\noindent \textbf{Cayley-Dickson doubling:} Let $(D,N_{D})$ be any composition algebra, $C=D\oplus D$ and $\lambda \in k^{\times}$. We define a multiplication and a norm $N_{C}$ on $C$ in the following way:

\begin{center}
$(x,y).(u,v)=(xu+\lambda \overline{v}y, vx+y\overline{u}), x,y,u,v\in D$;\\
\vspace{1 mm}

$N_{C}((x,y))=N_{D}(x)-\lambda N_{D}(y), x,y\in D.$
\end{center}

\noindent If $(D,N_{D})$ is associative, $(C,N_{C})$ is a composition algebra with respect to the multiplication defined above. And $(C,N_{C})$ is associative if and only if $(D,N_{D})$ is commutative and associative. The algebra $C$ contains $D$ as a composition subalgebra. Thus, from a given composition algebra, we can produce a composition algebra which has double dimension. This process is called the \textit{Cayley-Dickson doubling}. Conversely, any composition algebra can be obtained through this doubling process starting from the field $k$ (see \cite{SV}, Chapter $1$). If we change the constant $\lambda \in k^{\times}$ in the doubling process described above, the resultant composition algebra $C$ may change. In particular, for a given associative composition algebra $(D,N_{D})$, the set of all isomorphism classes of composition algebras $C$ obtained by Cayley-Dickson doubling from $D$, is in bijection with \[k^{\times}/\{N_{D}(x):x\in D^{\times}\}.\] In other words, the above set is in bijection with the composition algebras containing $D$ as a composition subalgebra (see \cite{SV}, Chapter $1$).

\vspace{2 mm}

Composition algebras exist in dimensions $1,2,4$ and $8$. The field $k$ itself is the only composition algebra of dimension $1$. In dimension $2$, the composition algebras are the quadratic extensions over $k$ and the quadratic algebra $k\times k$. Composition algebras with dimensions $4$ and $8$ are called \textit{quaternion} and \textit{octonion} algebras, respectively. Quaternion algebras are associative, but not commutative. Octonion algebras are neither commutative nor associative. So, the Cayley-Dickson doubling process starting from the field $k$ always generates a \textit{maximal flag} of composition algebras of the form $(k\subset K\subset Q\subset C)$, where $K$, $Q$, $C$ are composition algebras of dimensions $2$, $4$, $8$, respectively, each of which has been obtained from the earlier subalgebra by the doubling process.

\vspace{2 mm}

An \textit{isomorphism} between two composition algebras $(C_{1},N_{C_{1}})$ and $(C_{2},N_{C_{2}})$ is an isomorphism between the underlying algebras. This isomorphism gives rise to an isometry between the norms as well. Conversely, if the norms $N_{C_{1}}$ and $N_{C_{2}}$ are isometric, the composition algebras $C_{1}$ and $C_{2}$ are isomorphic. Up to isomorphism, there are unique split composition algebras in each of the dimensions $2,4$ and $8$. By an \textit{isomorphism} between two maximal flags $(k_{i}\subset K_{i}\subset Q_{i}\subset C_{i})$ for $i=1,2$, we mean an algebra isomorphism $\phi:C_{1}\rightarrow C_{2}$ such that $\phi(Q_{1})=Q_{2}$ and $\phi(K_{1})=K_{2}$. For more details on composition algebras, we refer the reader to (\cite{SV}, Chapter $1$).

\subsection{The split composition algebras} The split composition algebras in dimensions $2$ and $4$ are $k\times k$ and $M_{2}(k)$, respectively. The norms are the hyperbolic norms, i.e., the usual determinant for $M_{2}(k)$. The split octonion algebra can be described by defining a multiplication and norm on the following vector space of dimension $8$, \[Zorn(k)= \left(\begin{array}{cc}
k & k^{3} \\
k^{3} & k
\end{array}\right).\] The multiplication on $Zorn(k)$ we define by  \[\left(\begin{array}{cc}
  a & x \\
  y & b
  \end{array}\right). 
  \left(\begin{array}{cc}
  a^{\prime} & x^{\prime} \\
  y^{\prime} & b^{\prime}
  \end{array}\right) 
  =
  \left(\begin{array}{cc}
  aa^{\prime}+x^{t}y^{\prime} & ax^{\prime}+b^{\prime}x+y\times y^{\prime} \\
  a^{\prime}y+by^{\prime}+x\times x^{\prime} & bb^{\prime}+y^{t}x^{\prime} 
  \end{array}\right),\]

\noindent where $a,b,a^{\prime},b^{\prime}\in k$; $x,y,x^{\prime},y^{\prime}\in k^{3}$ and $`\times$' denotes the standard cross product on $k^{3}$. The space $Zorn(k)$ is a non-associative $k$-algebra with $\left(\begin{array}{cc}
1 & 0 \\
0 & 1
\end{array}\right)$
as the identity element, with respect to the above multiplication. This algebra is called the \textit{Zorn algebra} of vector matrices, which can be obtained by doubling the split quaternion algebra $M_{2}(k)$. The norm on $Zorn(k)$ is given by \[N(\left(\begin{array}{cc}
  a & x \\
  y & b
  \end{array}\right))= ab-x^{t}y, \text{ where } a,b\in k \text{ and } x,y \in k^{3}. \]

\subsection{Pfister forms and the discriminant} A \textit{Pfister form} over the field $k$ is a quadratic form of the form $\langle1,-a_{1}\rangle\otimes...\otimes \langle 1,-a_{n}\rangle$, for some $a_{1},..,a_{n}\in k^{\times}$. Sometimes we denote it by simply writing $\langle \langle a_{1},..,a_{n}\rangle\rangle$. Clearly, a Pfister form is non-degenerate and the underlying quadratic space has dimension $2^{n}$. If a Pfister form is isotropic, it is hyperbolic. The norms of the composition algebras are always Pfister forms. For more details, we refer to (\cite{LTY}, Chapter X).

\vspace{2 mm}

Let $K/k$ be a quadratic algebra and $(V,h)$ be a non-degenerate $K/k$-hermitian space of rank $n$ over $K$. Let $H=(h_{ij})_{n\times n}$ be the matrix representing $h$ with respect a fixed $K$-basis of $V$. Then we define the \textit{discriminant} of $h$ as \[\text{disc}(h)=(-1)^{\frac{n(n-1)}{2}}\text{det}(H).N(K/k)\in k^{\times}/N(K/k),\] where $N(K/k)\subset k^{\times}$ is the group of all norm values of the elements in $K^{\times}$ (see \cite{KMRT}, Chapter II). We say that $(V,h)$ is a \textit{trivial discriminant} hermitian form if $\text{det}(H)\in N(K/k)$. So, if $n=3$ and $(V,h)$ has trivial discriminant, then according to the above definition $\text{disc}(h)=(-1).N(K/k)$.

\subsection{The octonions and hermitian forms} Let $(C,N_{C})$ be an octonion algebra defined over the field $k$ and $K\subset C$ be a quadratic subalgebra. Let $K=k(\sqrt{a})$, for some $a\in k^{\times}$. Then $C$ is a $K$-module where $K$ acts on $C$ by the left multiplication, and we can define \[h(x,y)=b_{N_{C}}(x,y)+a^{-1}\sqrt{a}.b_{N_{C}}(\sqrt{a}x,y), \text{ for }x,y\in C,\] which is a non-degenerate $K/k$-hermitian form on $C$ (see \cite{JN} for details). If we have any $K$-submodule $D$ in $C$, the orthogonal complement $D^{\perp}\subset C$ with respect to the norm $N_{C}$ coincides with the orthogonal complement of $D$ with respect to the hermitian form $h$. Then $h_{1}$, the restriction of $h$ to $K^{\perp}\subset C$, is a trivial discriminant non-degenerate ternary $K/k$-hermitian form such that $C=K\oplus K^{\perp}$ and \[N_{C}((\alpha,x))=N(\alpha)+h_{1}(x,x), \forall \alpha\in K, x\in K^{\perp},\] (see \cite{JN}, \cite{TML} for details). Again, we assume that $Q\subset C$ is a quaternion subalgebra containing $K$, i.e., $Q$ is a $K$-submodule of $C$. Let $SL_{1}(Q)$ be the group of all elements in $Q$ with norm $1$. Then $SL_{1}(Q)\subset \text{Aut}(C)$ is the group of all algebra automorphisms of $C$ that fixes $Q$ pointwise (see \cite{SV}, Chapter $2$). If $h_{2}$ is the restriction of $h$ to $Q^{\perp}$, then for $y\in SL_{1}(Q)$ the map $x\mapsto yx$ ($\forall x\in Q^{\perp}$) is an element of $SU(h_{2})$ and we can easily see that (see \cite{JN}; \cite{SV}, Chapter $1,2$) \[SL_{1}(Q)\simeq SU(h_{2}).\] In particular, we can choose the matrix representative of $h_{1}$ to be of the form $\text{diag}(y_{1},y_{2},y_{3})$, where the same for $h_{2}$ is $\text{diag}(y_{2},y_{3})$ and $y_{1}y_{2}y_{3}$ is a $K/k$-norm. And the norm $N_{C}$ on the octonion $C$ is the trace form of the $K/k$-hermitian form $h$ given by $\text{diag}(1,y_{1},y_{2},y_{3})$, i.e., $N_{C}(x)=h(x,x),\forall x\in C$.

\subsection{Freudenthal algebras} For the rest of this section we assume that the field $k$ has characteristic different from $2$ and $3$. Let $(C,N_{C})$ be a composition algebra over the field $k$, and $M_{3}(C)$ be the space of all square matrices of order $3$ with entries from the algebra $C$. For $\Gamma=\text{diag}(\gamma_{1},\gamma_{2},\gamma_{3})\in GL_{3}(k)$, we can define an involution on $M_{3}(C)$ by $X^{\ast \Gamma}=\Gamma^{-1}\overline{X}^{t}\Gamma$, where $X=(x_{ij})_{3\times 3}\in M_{3}(C)$ and $\overline{X}=(\overline{x_{ij}})_{3\times 3}$. Let $\mathcal{H}_{3}(C,\Gamma)=\{X\in M_{3}(C): X^{\ast \Gamma}=X\}$ and we define a multiplication on this subspace of $M_{3}(C)$ by \[X.Y=\frac{1}{2}(XY+YX),\text{ for } X,Y\in \mathcal{H}_{3}(C,\Gamma),\] where $XY$ denotes the usual product of the matrices $X,Y\in M_{3}(C)$. Then with respect to the multiplication defined above $\mathcal{H}_{3}(C,\Gamma)$ forms an algebra with $I_{3}\in M_{3}(C)$ as the identity element, and these algebras are called \textit{reduced Freudenthal algebras}. All of these algebras contain zero divisors. We have $\text{dim}(\mathcal{H}_{3}(C,\Gamma))=3(\text{dim}(C)+1).$ In general, let $\mathcal{A}$ be an algebra over $k$ such that, for some field extension $K/k$, $\mathcal{A}\otimes K$ is isomorphic to a reduced Freudenthal algebra defined over the field $K$. Then $\mathcal{A}$ is called a \textit{Freudenthal algebra}. These algebras are cubic Jordan algebras and every element $X\in \mathcal{A}$ satisfies a cubic polynomial of the form \[X^{3}-T(X)X^{2}+S(X)X-N(X).I_{3}=0.\] The map $X\mapsto T(X)$ is called the \textit{generic trace} on $\mathcal{A}$. We can define a bilinear form $T_{\mathcal{A}}:\mathcal{A}\times \mathcal{A}\rightarrow k$ by $T_{\mathcal{A}}(X,Y)=T(X.Y)$, for $X,Y\in \mathcal{A}$, which is called the \textit{bilinear trace form} on $\mathcal{A}$. We can also define a quadratic map on $\mathcal{A}$ by $X\mapsto X^{\#}=X^{2}-T(X)X+S(X)$ for $X\in \mathcal{A}$, which is called the \textit{Freudenthal adjoint}. Let $\times :\mathcal{A}\times \mathcal{A}\rightarrow \mathcal{A}$ be the associated bilinear map defined by $X\times Y=(X+Y)^{\#}-X^{\#}-Y^{\#}$, for $X,Y\in \mathcal{A}$, which is called the \textit{Freudenthal cross product}. As we can see from the above equation, for all $X\in \mathcal{A}$ we have $X.X^{\#}=N(X).I_{3}$, where $N(X)$ is called the \textit{generic norm} of $X\in \mathcal{A}$. For more details, we refer to (\cite{KMRT}, Chapter IX) and (\cite{GPR}).

\section{The symplectic group and its representations}\label{3} In this section, we define the symplectic group of order $6$ and describe the representation of dimension $20$ for the same group, which decomposes into two irreducible representations. We also describe the $Sp_{6}$-invariant polynomials defined on the underlying vector space of the representation. 

\subsection{The symplectic group} Let $V_{6}$ be the vector space of dimension $6$ over the field $k$, with $\mathcal{B}=\{e_{1},e_{2},..,e_{6}\}$ as the standard basis. Then up to isometry there exists a unique non-degenerate alternating bilinear form $J$ on $V_{6}$, which is represented by the matrix \[ M_{J}= \left(\begin{array}{cc}
0 & I_{3} \\
-I_{3} & 0
\end{array}\right), \] with respect to the basis $\mathcal{B}$. We define the \textit{symplectic group} $Sp_{6}$ as the group of all isometries of $J$. So we get \[ Sp_{6}=\{g\in GL_{6}: gM_{J}g^{t}=M_{J}\}.\] We also define the group of all similitudes of the alternating form $J$ as the \textit{general symplectic group}, which is given by \[ GSp_{6}=\{g\in GL_{6}:gM_{J}g^{t}=\lambda(g)M_{J}, \lambda(g)\in GL_{1}\}.\]

\subsection{The representation} We consider the action of $Sp_{6}$ on $V=\wedge^{3}V_{6}$, induced by the action of $Sp_{6}$ on $V_{6}$. Then $V$ is a representation of $Sp_{6}$ defined over the field $k$, and we have \[V=\wedge^{3}V_{6}=\text{span} \{ e_{i}\wedge e_{j}\wedge e_{l}:1\leq i<j<l\leq 6\}.\] From now on, we will write $e_{i}e_{j}e_{l}$ for $e_{i}\wedge e_{j}\wedge e_{l}$. We have a natural contraction map $\psi : \wedge^{3}V_{6}\rightarrow V_{6}$ defined by \[\psi (v_{1}\wedge v_{2}\wedge v_{3})=J(v_{2},v_{3})v_{1}-J(v_{1},v_{3})v_{2}+J(v_{1},v_{2})v_{3},\text{ for }v_{1},v_{2},v_{3}\in V_{6},\] where $J$ is the non-degenerate alternating bilinear form defined on $V_{6}$ (see \cite{FH}, Lecture $17$). Clearly, $\psi$ is surjective and $X=\text{ker}(\psi)\subset V$ is invariant under the action of $Sp_{6}$. Then $\text{dim}(X)=14$ and $X$ is an irreducible representation of $Sp_{6}$ (see \cite{IJ}, \cite{YA1}). If we consider the $6$-dimensional subspace of $V$ spanned by the following subset \[\{e_{r_{1}}e_{2}e_{5}+e_{r_{1}}e_{3}e_{6},e_{r_{2}}e_{1}e_{4}+e_{r_{2}}e_{3}e_{6}, e_{r_{3}}e_{1}e_{4}+e_{r_{3}}e_{2}e_{5}: r_{1}=1,4; r_{2}=2,5;r_{3}=3,6\}\subset V,\] then the subspace is also $Sp_{6}$-invariant (see \cite{YA1}, Section $1$) and isomorphic to $V_{6}$ as a representation of $Sp_{6}$ through the above contraction $\psi$. It is straightforward to check that as a vector space $V$ is a direct sum of the above $6$-dimensional subspace and $X=\text{ker}(\psi)\subset V$. So, we get that as a $Sp_{6}$-representation $V=X\oplus V_{6}$, where $X=\text{ker}(\psi)$ is the irreducible representation of dimension $14$ and $V_{6}$ is the standard irreducible representation of dimension $6$ (see \cite{YA1}, Section $1$).

\vspace{2 mm}

\noindent\textbf{The identification of $V$ with $Z$:} We have $V=\wedge^{3}V_{6}=\text{span} \{ e_{i} e_{j} e_{l}:1\leq i<j<l\leq 6\}$, a vector space of dimension $20$. For $x\in V$, let us denote the coordinates by $x_{ijl}$ with respect to the above basis. We can choose an ordering of this basis of $V=\wedge^{3}V_{6}$ such that we can identify $V$ with $Z$ (described in Section \ref{1}) in the following way. $$V=\wedge^{3}V_{6}=\{(-x_{0},-y_{0},(a_{ij}),(b_{ij})):x_{0},y_{0} \text{ are scalars and } (a_{ij}),(b_{ij}) \text{ are }3\times 3 \text{ matrices}\},$$ where $x_{0}=-x_{123}, y_{0}=-x_{456}$ and the matrices $(a_{ij}),(b_{ij})$ are given by
\begin{center}
    
$(a_{ij})=\left(\begin{array}{ccc}
x_{423} & x_{143} & x_{124} \\
x_{523} & x_{153} & x_{125} \\
x_{623} & x_{163} & x_{126}
\end{array}\right)$ and $(b_{ij})=\left(\begin{array}{ccc}
x_{156} & x_{416} & x_{451} \\
x_{256} & x_{426} & x_{452} \\
x_{356} & x_{436} & x_{453}
\end{array}\right).$

\end{center}

\vspace{1 mm}

Let $\text{char}(k)\neq 2,3$, and we consider the Jordan algebra $M_{3}(k)^{+}\simeq \mathcal{H}_{3}(k\times k, I_{3})$ with the underlying vector space as $M_{3}(k)$ (see \cite{KMRT}, Chapter IX). We have \[Z=\left(\begin{array}{cc}
k & M_{3}(k)^{+} \\
M_{3}(k)^{+} & k
\end{array}\right)\] and we can get an algebra structure on $Z$ by defining an appropriate multiplication on it using the Jordan algebra structure of $M_{3}(k)^{+}$, which makes $Z$ a structurable algebra of skew-dimension $1$ (see \cite{AF}, \cite{BM} for details). As a vector space (also as a $Sp_{6}$-representation) we can identify $V$ with $Z$ from the above description of $V$ by assigning \[(-x_{0},-y_{0},(a_{ij}),(b_{ij}))\in V \text{ to }\left(\begin{array}{cc}
-x_{0} & (a_{ij}) \\
(b_{ij}) & -y_{0}
\end{array}\right) \in Z.\] We can consider this identification between two $Sp_{6}$-representations, in the case when $\text{char}(k)=3$ as well (avoiding the algebra structure on $Z$).

\subsection{The invariants} Now we will describe the $Sp_{6}$-invariant polynomials defined on $V$. For $1\leq m\leq 6$, we define the operator $\partial/\partial e_{m}$ on $V$ by \[ \frac{\partial}{\partial e_{m}}(e_{i}\wedge e_{j}\wedge e_{l})=\delta_{im} e_{j}\wedge e_{l} - \delta_{jm}e_{i}\wedge e_{l} + \delta_{lm} e_{i}\wedge e_{j},\] where $\delta_{rs}$ is the Kronecker delta. Let $\tau =e_{1}\wedge e_{2}\wedge e_{3}\wedge e_{4}\wedge e_{5} \wedge e_{6}$. If $\tilde{x} \in V$, then we can write \[\tilde{x}=\sum_{i,j,l} x_{ijl}e_{i}\wedge e_{j}\wedge e_{l}.\] For $1\leq i,j\leq 6$, we define a homogeneous polynomial $\phi_{ij}$ of degree $2$ on $V$ by \[ \tilde{x}\wedge \frac{\partial}{\partial e_{i}}(\tilde{x})\wedge e_{j} = \phi_{ij}(\tilde{x})\tau.\] If we take $\phi(\tilde{x})=(\phi_{ij}(\tilde{x}))_{6\times 6}$, then we have the following result from (\cite{SK}, Proposition $7$, Section $5$).

\begin{proposition}\label{P2.1} Let us consider $V=\wedge^{3}V_{6}$ as the $GL_{6}$-representation with the obvious action. Then 

    \begin{enumerate}
        \item $\phi(g.\tilde{x})=\text{det}(g). g\phi(\tilde{x})g^{-1}$ for $g\in GL_{6}$;
        \item $\phi(\tilde{x})^{2}=f(\tilde{x}).I_{6}$, for some irreducible polynomial $f(\tilde{x})$ of degree $4$. Moreover, $f(\tilde{x})$ is a relative invariant on $\wedge^{3}V_{6}$, corresponding to the character $(\text{det}(g))^{2}$ for $g\in GL_{6}$.
    \end{enumerate}
\end{proposition}

As we can see from the above proposition, $f$ is an absolute invariant defined on $V$ under the $Sp_{6}$-action, as $\text{det}(g)=1,\forall g\in Sp_{6}$. If we take \[f_{1}=(-\frac{1}{4}f)|_{X},\] it is the same $Sp_{6}$-invariant quartic homogeneous irreducible polynomial defined on $X$ as mentioned in (\cite{IJ}), (\cite{PS}). Now we define the polynomial $f_{2}$ on $V$ by \[f_{2}(\tilde{x})=(-\frac{1}{2})v^{t}M_{J}\phi(x) v, \text{ for } \tilde{x}=(x,v)\in V=X\oplus V_{6}.\] We can directly check that $f_{2}$ is again an absolute invariant on $V$ under the action of the group $Sp_{6}$ (see Proposition \ref{P2.1}). Finally, we get the following proposition for the PV $(Sp_{6}\times GL_{1}^{2},V)$ from (\cite{SKa}, Proposition $5.2$).

\begin{proposition}
    The PV $(Sp_{6}\times GL_{1}^{2},V)$ has two irreducible relative invariants $f_{1}$ of degree $4$ and $f_{2}$ of degree $4$. Let $\chi_{i}$ be the characters corresponding to $f_{i}$, for $i=1,2$. Then we have $\chi_{1}((g,a,b))=a^{4}$ and $\chi_{2}((g,a,b))=a^{2}b^{2}$, for $g\in Sp_{6}$, $a,b\in GL_{1}$. 
\end{proposition}

We will discuss the PV $(Sp_{6}\times GL_{1}^{2},V)$ in detail later in Section \ref{5}, and prove that the set of all points in $V$ where both $f_{1}$ and $f_{2}$ are non-zero forms a single orbit.

\vspace{1 mm}

\begin{remark}
    The degree $4$ polynomial $f$ in Proposition \ref{P2.1} is the same invariant defined on the structurable algebra $Z$ (see \cite{AB}, \cite{AF}, \cite{LM}), if we use the identification of $V$ with $Z$ in the previous subsection. Even $f$ is an absolute invariant under the induced action of $SL_{6}\subset GL_{6}$ which we can see from Proposition \ref{P2.1} (see \cite{LM} also). For more details on $f$, we refer to (\cite{SKa}) and (\cite{BGL}).
\end{remark}

\vspace{1 mm}

We will write $V^{ss}=\{\tilde{x}=(x,v)\in V=X\oplus V_{6}:f_{1}(x)\neq 0, f_{2}(\tilde{x})\neq 0\}\subset V$. This is a $Sp_{6}$-invariant subset of $V$, as both $f_{1}$ and $f_{2}$ are $Sp_{6}$-invariant.

\section{Main theorems}\label{4} In this section, we compute the orbit space of the Zariski-dense subset $V^{ss}\subset V=\wedge^{3}V_{6}$, under the $Sp_{6}$-action. In (\cite{IJ}), the orbit decompositions of $X$ have been discussed in detail, and in (\cite{PS}) we have seen that the orbit space gives us a parametrization of all composition algebras defined over the underlying field $k$. Here we will see a description of all maximal $4$-flags of composition algebras, in terms of the $Sp_{6}(k)$-orbits in $V^{ss}(k)\subset V(k)$.

\vspace{2 mm}

We have $V^{ss}=\{\tilde{x}=(x,v)\in V=X\oplus V_{6}: f_{1}(x)\neq 0, f_{2}(\tilde{x})\neq 0\}$. From now on, we will write $G$ for the group $Sp_{6}$. We start with the following Proposition and its Corollary from (\cite{SJP}, Chapter I, Section $5$, Proposition $36$ and Corollary $1$), which is required to proceed further.

\begin{proposition}
    Let $H_{1}$ and $G_{1}$ be two algebraic groups defined over the field $k$ and $H_{1}\subset G_{1}$. Then the sequence of cohomology sets \[1\rightarrow H^{0}(k,H_{1})\rightarrow H^{0}(k,G_{1})\rightarrow H^{0}(k,G_{1}/H_{1})\rightarrow H^{1}(k, H_{1})\rightarrow H^{1}(k,G_{1})\] is exact.
\end{proposition}

\begin{corollary}\label{c1}
    The kernel of $H^{1}(k,H_{1})\rightarrow H^{1}(k,G_{1})$ may be identified with the quotient space of $H^{0}(k,G_{1}/H_{1})=(G_{1}/H_{1})(k)$ by the action of the group $H^{0}(k,G_{1})=G_{1}(k)$.
\end{corollary}

Let $A$, $A_{1}$ be two algebras over the field $k$, and $B\subset A$, $B_{1}\subset A_{1}$ be their subalgebras, respectively. We say that the $2$-flags of algebras $(B\subset A)$ and $(B_{1}\subset A_{1})$ are isomorphic, if $\exists$ an algebra isomorphism $\phi :A\rightarrow A_{1}$ such that $\phi(B)=B_{1}$. Now we consider the irreducible $G$-representation $X\subset V=X\oplus V_{6}$ and take $X^{ss}=\{x\in X:f_{1}(x)\neq 0\}\subset X$. Then each point in the orbit space $X^{ss}_{k}/G_{k}$ represents an isomorphism class of a unique octonion algebra $C$ along with a unique quadratic subalgebra $K\subset C$, and we get all such $2$-flags of composition algebras $(K\subset C)$ over the field $k$ in this way; which we prove in the following result (see \cite{PS} also).

\begin{theorem}\label{KC}
     There is a surjection $f^{ss}:X^{ss}_{k}/G_{k}\rightarrow \{$Isomorphism classes of the $2$-flags of composition algebras $(K\subset C): \text{dim}(K)=2, \text{dim}(C)=8\}$.
    
\end{theorem}

\begin{proof}
    Let $x\in X^{ss}_{k}$ and $f_{1}(x)=i\in k^{\times}$. We first assume that $-i\notin k^{\times 2} $. Then $x$ is $G_{k}$-equivalent to an element of the form \[y=-e_{1}e_{2}e_{3}-y_{0}e_{4}e_{5}e_{6}+y_{1}e_{1}e_{5}e_{6}+y_{2}e_{4}e_{2}e_{6}+y_{3}e_{4}e_{5}e_{3},\] where $y_{0},y_{1},y_{2},y_{3}\in k$ and $f_{1}(y)=y_{1}y_{2}y_{3}-(1/4)y_{0}^{2}=i$ (see \cite{IJ}, page 1022, 1028). Again, over $k(\sqrt{-i})$ the element $x$ is equivalent to the element \[x_{1}=-e_{1}e_{2}e_{3}-2(\sqrt{-i})e_{4}e_{5}e_{6},\] and we can choose an element $g\in G({k(\sqrt{-i})})$ such that $gx_{1}=y$ (see \cite{IJ}, page 1023). So, over $k(\sqrt{-i})$, $f_{1}^{-1}\{i\}$ is a single orbit. Let $H\subset Sp_{6}$ be the stabilizer of the element $x_{1}\in X^{ss}$, and as we can see in (\cite{IJ}, page 1024) \[H = \biggl\{\left(\begin{array}{cc}
A & 0 \\
0 & (A^{t})^{-1}
\end{array}\right):A\in SL_{3}\biggr\}\simeq SL_{3}.\] Then the stabilizer of $y$ is isomorphic to $gHg^{-1}\simeq SL_{3}$ over $k(\sqrt{-i})$ and the $k$-rational points in the stabilizer of $y$ are given by the relation \[\sigma \biggl( g\left(\begin{array}{cc}
A & 0 \\
0 & (A^{t})^{-1}
\end{array}\right)g^{-1}\biggr)=g\left(\begin{array}{cc}
A & 0 \\
0 & (A^{t})^{-1}
\end{array}\right)g^{-1}, A\in SL_{3},\] where $\sigma$ is the non-trivial involution on $k(\sqrt{-i})$. As we can compute the element $g$ (see \cite{IJ}, page 1023), a straightforward calculation shows that $\text{Stab}(x)\simeq \text{Stab(y)}$ is isomorphic to the special unitary group $SU(h)$ over $k$, where $h$ is the non-degenerate ternary $k(\sqrt{-i})/k$-hermitian form represented by $\text{diag}(y_{1},y_{2},y_{3})$. The hermitian form $h$ has trivial discriminant as $y_{1}y_{2}y_{3}=(1/4)y_{0}^{2}+i=N((1/2)y_{0}+\sqrt{-i})$, where $N$ is the $k(\sqrt{-i})/k$-norm. Hence from Corollary \ref{c1} we get \[ f_{1}^{-1}\{i\}(k)/G_{k}\longleftrightarrow H^{1}(k,SU(h)), \] as $H^{1}(k,Sp_{6})$ is trivial. Now, $H^{1}(k,SU(h)) $ has one to one correspondence with the set of all isometry classes of trivial discriminant non-degenerate ternary $k(\sqrt{-i})/k$-hermitian forms (see \cite{KMRT}, Chapter VII, Example $(29.19)$), which classify the isomorphism classes of octonion algebras over the field $k$ containing $k(\sqrt{-i})$ as subalgebra (see \cite{TML}, Theorem $2.2$). So we define \[f^{ss}(\mathcal{O}(x))=(k(\sqrt{-i})\subset C),\]
where $f_{1}(x)=i\in k^{\times}, -i \notin k^{\times 2}$, $\mathcal{O}(x)$ denotes the $G_{k}$-orbit of $x\in X^{ss}_{k}$ and $C$ is the octonion determined by the $k(\sqrt{-i})/k$-hermitian form represented by $\text{diag}(y_{1},y_{2},y_{3})$ (see \cite{JN}, \cite{TML}).

\vspace{2 mm}

 If $ f_{1}(x)=i$ and $-i\in k^{\times 2}$, then $f_{1}^{-1}\{i\}(k)$ is a single $G_{k}$-orbit and $k(\sqrt{-i})= k[t]/ \langle t^{2}+i\rangle \simeq k\times k$. There is a unique octonion algebra containing $k\times k$ as a subalgebra, namely $Zorn(k)$. So, in this case we define \[ f^{ss}(\mathcal{O}(x))=(k\times k, Zorn(k)),\] and this makes $f^{ss}$ well-defined on $X^{ss}_{k}/G_{k}$.

 \vspace{2 mm}

 In order to prove the surjectivity of $f^{ss}$, let $(K\subset C)$ be a $2$-flag as described earlier. Then $K\simeq k(\sqrt{-i})$, for some $i\in k^{\times}$ (as char$(k)\neq 2$), and we get a unique trivial discriminant non-degenerate ternary $k(\sqrt{-i})/k$-hermitian form $h$ which corresponds to $C$ (see \cite{JN}, \cite{TML}). If the diagonal matrix $\text{diag}(y_{1},y_{2},y_{3})$ represents $h$, and we choose $y_{0}\in k$ satisfying $y_{1}y_{2}y_{3}=(1/4)y_{0}^{2}+i$, then $f^{ss}(\mathcal{O}(x))=(K\subset C)$, where \[x=-e_{1}e_{2}e_{3}-y_{0}e_{4}e_{5}e_{6}+y_{1}e_{1}e_{5}e_{6}+y_{2}e_{4}e_{2}e_{6}+y_{3}e_{4}e_{5}e_{3}\in X^{ss}_{k}.\] Hence we get the required surjection.

\end{proof}

\begin{remark}
    If $x\in X^{ss}_{k}$ with $f_{1}(x)=i\in k^{\times}$, $-i\notin k^{\times 2}$ and $h$ is the trivial discriminant $k(\sqrt{-i})/k$-hermitian form of rank $3$ such that $\text{Stab}(x)\simeq SU(h)$, then $\text{Aut}(C/k(\sqrt{-i}))=SU(h)$, where $f^{ss}(\mathcal{O}(x))=(k(\sqrt{-i})\subset C)$ for $\mathcal{O}(x)\in X^{ss}_{k}/G_{k}$ (see \cite{JN}, Theorem $3$). If $-i\in k^{\times 2}$, then $f^{ss}(\mathcal{O}(x))=( k\times k \subset Zorn(k))$ and $\text{Stab}(x)\simeq SL_{3}=\text{Aut}(Zorn(k)/k\times k)$.
\end{remark}

\vspace{2 mm}

\begin{remark}
    Let $(K\subset C)$ be a $2$-flag of an octonion $C$ and it's quadratic subalgebra $K=k(\sqrt{-i})$, $-i\in k^{\times}-k^{\times 2}$. Then $(f^{ss})^{-1}\{(K\subset C)\}$ is the set of all orbits $\mathcal{O}(x)\in X^{ss}_{k}/G_{k}$ with $f_{1}(x)=ia^{2}$ for some $a\in k^{\times}$, where \[x=-e_{1}e_{2}e_{3}-y_{0}e_{4}e_{5}e_{6}+y_{1}e_{1}e_{5}e_{6}+y_{2}e_{4}e_{2}e_{6}+y_{3}e_{4}e_{5}e_{3}\in X^{ss}_{k},\] and the trivial discriminant ternary $K/k$-hermitian form $h$ represented by the non-singular diagonal matrix $\text{diag}(y_{1},y_{2},y_{3})$ is such that $ N_{C}((a,y))=N(a)+h(y,y), \text{ for } a\in K,y\in K^{\perp}\subset C$ and $N$ is the norm on $K$ (see Section \ref{2}). If $-i\in k^{\times 2}$, i.e., $K\simeq k\times k$, then $(f^{ss})^{-1}\{(K\subset C)\}$ is the set of all orbits $\mathcal{O}(x)\in X^{ss}_{k}/G_{k}$ with $-f_{1}(x)\in k^{\times 2}$.
\end{remark}

\vspace{2 mm}

\begin{remark}

The action of $G$ on $V=X\oplus V_{6}$ is given by $g.(x,v)=(gx,gv)$ for $g\in G$ and $x\in X,v\in V_{6}$. So, if $\tilde{x}=(x,v)\in V^{ss}_{k}$ and $f_{1}(x)=i\in k^{\times}$, then for $g\in G_{k}$ we have $g.\tilde{x}=(g.x,g.v)$ and the $X$-component $g.x$ of $g.\tilde{x}$ represents the same $2$-flag $(K\subset C)$ as the $X$-component of $\tilde{x}$, according to the above result.
\end{remark}

\vspace{2 mm}

Before proceeding further, we first discuss some results on the other $G$-invariant polynomial $f_{2}$ defined on $V$ in Section \ref{3}. As we have already mentioned, any $x\in X^{ss}_{k}$ with $f_{1}(x)=i\in k^{\times}$ can be reduced to an element of the form $-e_{1}e_{2}e_{3}-y_{0}e_{4}e_{5}e_{6}$ over $k(\sqrt{-i})$, where $y_{0}=2\sqrt{-i}$ (see \cite{IJ} for details). Since $f_{2}$ is $G$-invariant, it suffices to compute $f_{2}$ for the points of the form $(-e_{1}e_{2}e_{3}-y_{0}e_{4}e_{5}e_{6},v)\in V^{ss}$ over $k(\sqrt{-i})$.

\begin{proposition}\label{inv2}
 If $\tilde{x}=(-e_{1}e_{2}e_{3}-y_{0}e_{4}e_{5}e_{6}, v)\in V^{ss}$, we have $f_{2}(\tilde{x})=-y_{0}(v_{1}v_{4}+v_{2}v_{5}+v_{3}v_{6}),$ where $v=(v_{1},..,v_{6})^{t}\in V_{6}$.
\end{proposition}

\begin{proof}
We have $f_{2}(\tilde{x})=(-\frac{1}{2})v^{t}M_{J}\phi(x) v$ (see Section \ref{3}). For $x=-e_{1}e_{2}e_{3}-y_{0}e_{4}e_{5}e_{6}$ we can directly compute \[M_{J}\phi(x)= \left(\begin{array}{cc}
0 & y_{0}I_{3} \\
y_{0}I_{3} & 0
\end{array}\right), \] (see Section \ref{3}). Therefore, we get $f_{2}(\tilde{x})=-y_{0}(v_{1}v_{4}+v_{2}v_{5}+v_{3}v_{6})$.

\end{proof}

\vspace{1 mm}

\begin{remark}
    As $f_{2}$ is an absolute invariant of $Sp_{6}$ (see Section \ref{3}), we get $\text{Stab}(-e_{1}e_{2}e_{3}-y_{0}e_{4}e_{5}e_{6})\simeq SL_{3}\subset SO(q)$, where $q$ is the hyperbolic quadratic form defined on $V_{6}$ by \[q(v)=(v_{1}v_{4}+v_{2}v_{5}+v_{3}v_{6}), \text{ for }v=(v_{1},..,v_{6})^{t}\in V_{6}.\]
\end{remark}

\vspace{2 mm}

\begin{proposition}\label{f2}
    If $(x,v)\in V^{ss}_{k}$ with $v=(v_{1},v_{2},v_{3},v_{4},v_{5},v_{6})^{t}\in V_{6}$ and $x=-e_{1}e_{2}e_{3}-y_{0}e_{4}e_{5}e_{6}+y_{1}e_{1}e_{5}e_{6}+y_{2}e_{4}e_{2}e_{6}+y_{3}e_{4}e_{5}e_{3}\in X^{ss}_{k}$, then \[f_{2}((x,v))=-(y_{2}y_{3}v_{1}^{2}+y_{3}y_{1}v_{2}^{2}+y_{1}y_{2}v_{3}^{2}+y_{1}v_{4}^{2}+y_{2}v_{5}^{2}+y_{3}v_{6}^{2}+y_{0}v_{1}v_{4}+y_{0}v_{2}v_{5}+y_{0}v_{3}v_{6}).\]
\end{proposition}

\begin{proof}
    The proof follows from direct computations, using the definition of $f_{2}$ (see Section \ref{3}).
\end{proof}

Let $\tilde{x}=(x,v)\in V^{ss}_{k}$; i.e., $f_{1}(x)=i\in k^{\times}$ and also $f_{2}(\tilde{x})\neq 0$. Assume that $-i\notin k^{\times 2}$. Then we have the following result.

\begin{proposition}\label{SL}
    The stabilizer of $\tilde{x}=(x,v)\in V^{ss}_{k}$ is isomorphic to $SL_{2}$ over $k(\sqrt{-i})$, where $f_{1}(x)=i\neq 0$ and $-i\notin k^{\times 2}$.
\end{proposition}

\begin{proof}
    We have $\tilde{x}=(x,v)\in V=X\oplus V_{6}$ and $f_{1}(x)=i\in k^{\times}, f_{2}(\tilde{x})\in k^{\times}$. As we have seen in (\cite{IJ}), $\exists $ $g\in G(k(\sqrt{-i}))$ such that \[g.x=-e_{1}e_{2}e_{3}-2(\sqrt{-i})e_{4}e_{5}e_{6}\in X.\] So, we may assume $\tilde{x}=( -e_{1}e_{2}e_{3}-2(\sqrt{-i})e_{4}e_{5}e_{6},v)$ up to $G(k(\sqrt{-i}))$-equivalence. The stabilizer of $-e_{1}e_{2}e_{3}-2(\sqrt{-i})e_{4}e_{5}e_{6}$ in $G$ is given by \[ \biggl\{\left(\begin{array}{cc}
A & 0 \\
0 & (A^{t})^{-1}
\end{array}\right):A\in SL_{3}\biggr\}\simeq SL_{3}.\] Hence, the stabilizer of $\tilde{x}$ in $G(k(\sqrt{-i}))$ is contained in the above copy of $SL_{3}$, as it will stabilize both $-e_{1}e_{2}e_{3}-2(\sqrt{-i})e_{4}e_{5}e_{6}$ and $v$ individually. Let $v=(v_{1},v_{2},v_{3},v_{4},v_{5},v_{6})^{t}\in V_{6}$. Then $f_{2}(\tilde{x})\neq 0$ implies $v_{1}v_{4}+v_{2}v_{5}+v_{3}v_{6}\neq 0$ (see Proposition \ref{inv2}). Without loss of generality, we may assume that $v_{1}\neq 0, v_{4}\neq 0$. To see this, suppose $v_{2}\neq 0,v_{5}\neq 0$ for the above $v\in V_{6}$. Then we have \[g^{\prime}.((-e_{1}e_{2}e_{3}-2(\sqrt{-i})e_{4}e_{5}e_{6},v))=(-e_{1}e_{2}e_{3}-2(\sqrt{-i})e_{4}e_{5}e_{6},(-v_{2},v_{1},v_{3},-v_{5},v_{4},v_{6})^{t}),\] where $g^{\prime}\in G$ is given by \[g^{\prime}=\left(\begin{array}{cc}
A_{1} & 0 \\
0 & (A_{1}^{t})^{-1}
\end{array}\right) \text{ and }A_{1}=\left(\begin{array}{ccc}
0 & -1 & 0 \\
1 & 0 & 0 \\
0 & 0 & 1
\end{array}\right)\in SL_{3}.\] The same can be done if $v_{3}\neq 0, v_{6}\neq 0$, by changing the above $A_{1}$ in $g^{\prime}$. So, we assume that $v_{1},v_{4}\neq 0$. Now we have \[A=\left(\begin{array}{ccc}
v_{1}^{-1} & 0 & 0 \\
-v_{2} & v_{1} & 0 \\
-v_{3}v_{1}^{-1} & 0 & 1
\end{array}\right) \in SL_{3} \text{ and } g_{1}=\left(\begin{array}{cc}
A & 0 \\
0 & (A^{t})^{-1}
\end{array}\right)\in G. \] Then $g_{1}.v=(1,0,0,v_{4}^{\prime},v_{5}^{\prime},v_{6}^{\prime})^{t}$ for some scalars $v_{4}^{\prime},v_{5}^{\prime},v_{6}^{\prime}$. Here $v_{4}^{\prime}=v_{1}v_{4}+v_{2}v_{5}+v_{3}v_{6}\neq 0$. So, we take \[B=\left(\begin{array}{ccc}
v_{4}^{\prime-1} & 0 & 0 \\
-v_{5}^{\prime} & v_{4}^{\prime} & 0 \\
-v_{6}^{\prime}v_{4}^{\prime -1} & 0 & 1
\end{array}\right) \in SL_{3} \text{ and } g_{2}=\left(\begin{array}{cc}
(B^{t})^{-1} & 0 \\
0 & B
\end{array}\right)\in G.\] Then we get $g_{2}.(g_{1}.v)=(q(v),0,0,1,0,0)^{t}$, where $q(v)=v_{1}v_{4}+v_{2}v_{5}+v_{3}v_{6}$. As both $g_{1},g_{2}$ belong to the stabilizer of $-e_{1}e_{2}e_{3}-2(\sqrt{-i})e_{4}e_{5}e_{6}$, we have \[ g_{2}.(g_{1}.\tilde{x})=(-e_{1}e_{2}e_{3}-2(\sqrt{-i})e_{4}e_{5}e_{6},(q(v),0,0,1,0,0)^{t} ),\] and the entire process is rational over $k(\sqrt{-i})$. The stabilizer of $(q(v),0,0,1,0,0)^{t}$ inside the stabilizer of $-e_{1}e_{2}e_{3}-2(\sqrt{-i})e_{4}e_{5}e_{6}$ contains elements of the form \[ g_{3}=\left(\begin{array}{cc}
A & 0 \\
0 & (A^{t})^{-1}
\end{array}\right)\in Sp_{6}, \text{ where } A=\left(\begin{array}{ccc}
1 & a & b \\
0 & c & d \\
0 & e & f
\end{array}\right)\in SL_{3}, \] and $a,b,c,d,e,f$ are scalars satisfying $cf-ed=1$. Using the relation \[g_{3}.(q(v),0,0,1,0,0)^{t}=(q(v),0,0,1,0,0)^{t}\] we get $a=b=0$. So, the stabilizer is given by \[ SL_{2}\simeq \biggl\{\left(\begin{array}{cc}
C &  0 \\
0  &  (C^{t})^{-1}
\end{array}\right): C= \left(\begin{array}{cc}
1 & 0 \\
0 & C_{1}
\end{array}\right), C_{1}\in SL_{2}\biggr\}\subset \biggl\{\left(\begin{array}{cc}
A & 0 \\
0 & (A^{t})^{-1}
\end{array}\right):A\in SL_{3}\biggr\}.\] Thus, the result follows.

\end{proof}

\vspace{1 mm}

\begin{remark}
    If $-i\in k^{\times 2}$ in the above result, the stabilizer of $\tilde{x}\in V^{ss}_{k}$ is isomorphic to $SL_{2}$ over the base field $k$ itself. 
\end{remark}

\vspace{2 mm}

\begin{remark}\label{quat}
From the above result, it follows that the stabilizer of any element $\tilde{x}\in V^{ss}_{k}$ is isomorphic to some $k(\sqrt{-i})/k$-form of $SL_{2}$ over $k$, i.e., $SL_{1}(Q)=\{u\in Q:N_{Q}(u)=1\}$ for some unique quaternion algebra $(Q,N_{Q})$ over $k$ (see \cite{SJP}, Chapter III, Section $1$, page 125).    
\end{remark}

\vspace{2 mm}

\begin{theorem}\label{fiber orbit}
    The subset $f_{1}^{-1}\{i\}\cap V^{ss}\subset V $ is $G$-invariant for each non-zero scalar $i$. If $-i\in k^{\times}-k^{\times 2}$, over $k(\sqrt{-i})$ this set decomposes into the disjoint orbits with representatives \[(-e_{1}e_{2}e_{3}-2(\sqrt{-i})e_{4}e_{5}e_{6}, (a,0,0,1,0,0)^{t})\in V^{ss}, a\neq 0.\]
\end{theorem}

\begin{proof}
    As both $f_{1}$ and $V^{ss}$ are invariant under the $G$-action, the first claim is true. Now, from the proof of the Proposition \ref{SL} we get that any element $\tilde{x}=(x,v)\in V^{ss}$ with $f_{1}(x)=i\neq 0$ is $G(k(\sqrt{-i}))$-equivalent to an element of the form \[\tilde{y}=(-e_{1}e_{2}e_{3}-2(\sqrt{-i})e_{4}e_{5}e_{6}, (q(v),0,0,1,0,0)^{t})\in V^{ss},\] where $q$ is the hyperbolic quadratic form defined on $V_{6}$ by \[ q(v)=v_{1}v_{4}+v_{2}v_{5}+v_{3}v_{6}, \text{ for } v=(v_{1},v_{2},v_{3},v_{4},v_{5},v_{6})^{t}\in V_{6}.\] Since $q$ is hyperbolic, it is universal and we also have $\text{Stab}(-e_{1}e_{2}e_{3}-2(\sqrt{-i})e_{4}e_{5}e_{6})\simeq SL_{3}\subset SO(q)$. So, all the orbits are given by the following representatives \[ (-e_{1}e_{2}e_{3}-2(\sqrt{-i})e_{4}e_{5}e_{6}, (a,0,0,1,0,0)^{t}),\] where $a$ is any non-zero scalar.
\end{proof}

\vspace{1 mm}

\begin{remark}
    If $-i\in k^{\times 2}$ in the above theorem, the $G_{k}$-orbits in $f_{1}^{-1}\{i\}(k)\cap V^{ss}_{k}$ are given by the same representatives, i.e., \[(-e_{1}e_{2}e_{3}-2(\sqrt{-i})e_{4}e_{5}e_{6}, (a,0,0,1,0,0)^{t})\in V^{ss}_{k}, a\in k^{\times}.\]
\end{remark}

\vspace{2 mm}

    For any $i,j\in k^{\times }$ we take $U_{ij}=f_{1}^{-1}\{ i\}\cap f_{2}^{-1}\{j\}\subset V^{ss}$, i.e., $U_{ij}$ is $G$-invariant. Let $SL_{1}(Q)$ be the stabilizer of an element $(x,v)\in U_{ij}(k)$ under the action of $G_{k}$, for some quaternion algebra $Q$ over $k$ (see Remark \ref{quat}). Then we have

\begin{corollary}\label{CUij}
    If $-i\notin k^{\times 2}$, $U_{ij}$ is a single $G$-orbit over $k(\sqrt{-i})$, and $U_{ij}(k)/G_{k}$ has bijection with the set of all isomorphism classes of octonion algebras containing $Q$ as a quaternion subalgebra.
\end{corollary}

\begin{proof}
    The first part follows from Theorem \ref{fiber orbit}, as any element $(x,v)\in U_{ij}$ is $G$-equivalent to \[(-e_{1}e_{2}e_{3}-2(\sqrt{-i})e_{4}e_{5}e_{6},(\frac{-j}{2\sqrt{-i}},0,0,1,0,0)^{t})\] over $k(\sqrt{-i})$. Now, using Corollary \ref{c1} we get \[U_{ij}(k)/G_{k} \longleftrightarrow H^{1}(k,SL_{1}(Q)).\] The cohomology set $H^{1}(k,SL_{1}(Q))$ is in bijection with the isomorphism classes of octonion algebras containing $Q$ as a subalgebra.
\end{proof}

\vspace{2 mm}

\begin{remark}
   $U_{ij}(k)$ is a single $G_{k}$-orbit if $-i\in k^{\times 2}$. So, from the above corollary we get the $k$-rational orbit space $V^{ss}_{k}/G_{k}$, as $V^{ss}_{k}=\bigsqcup_{i,j\in k^{\times}}U_{ij}(k)$ and each $U_{ij}(k)$ is invariant under the action of $G_{k}$. 
\end{remark}

\vspace{1 mm}

Now we proceed to the main theorem of this section, where we construct the canonical surjection from the orbit space $V^{ss}_{k}/G_{k}$ onto the set of all maximal flags of composition algebras defined over the field $k$.

\begin{theorem}\label{CD}
Each point in the orbit space $V^{ss}_{k}/G_{k}$ represents a unique isomorphism class of maximal flags of composition algebras over $k$, and all such isomorphism classes of maximal $4$-flags can be obtained in this way.
\end{theorem}

\begin{proof}
It will be enough to define a surjective map $f^{ss}_{1}$ from the orbit space $V^{ss}_{k}/G_{k}$ onto the set of all isomorphism classes of maximal $4$-flags of composition algebras over $k$.

\vspace{2 mm}

Let $\tilde{x}=(x,v)\in V^{ss}_{k}$. So, $f_{1}(x)=i$ for some $i\in k^{\times}$, and $f_{2}(\tilde{x})\in k^{\times}$. We first consider the case where $-i\notin k^{\times 2}$. We may assume that $x$ is of the form \[ x=-e_{1}e_{2}e_{3}-y_{0}e_{4}e_{5}e_{6}+y_{1}e_{1}e_{5}e_{6}+y_{2}e_{4}e_{2}e_{6}+y_{3}e_{4}e_{5}e_{3},\] where $y_{1}y_{2}y_{3}=\frac{1}{4}y_{0}^{2}+i$. Then $\mathcal{O}(x)\in X^{ss}_{k}/G_{k}$ determines a unique $2$-flag of the form $(K\subset C)$, as we can see in Theorem \ref{KC}, where $K=k(\sqrt{-i})$. Let $Q$ be the quaternion defined over $k$ such that $SL_{1}(Q)$ is the stabilizer of $\tilde{x}\in V^{ss}_{k}$ in $Sp_{6}(k)$, and so we have $SL_{1}(Q)\subset SU(h)$, where $h$ is the non-degenerate ternary $k(\sqrt{-i})/k$-hermitian form represented by the diagonal matrix $\text{diag}(y_{1},y_{2},y_{3})$ (see Theorem \ref{KC}). Then we must have $K=k(\sqrt{-i})\subset Q\subset C$ (as $SL_{1}(Q)\subset SU(h)= \text{Aut}(C/K)\subset \text{Aut}(C)$), and we define \[f_{1}^{ss}(\mathcal{O}(\tilde{x}))=(k\subset K\subset Q\subset C),\] which is a unique maximal flag of composition algebras determined by the orbit $\mathcal{O}(\tilde{x})\in V^{ss}_{k}/G_{k}$. If $-i\in k^{\times 2}$, $k(\sqrt{-i})= k[t]/\langle t^{2}+i\rangle\simeq k\times k$ and $f_{1}^{ss}(\mathcal{O}(\tilde{x}))$ can be defined similarly using Theorem \ref{KC} and Proposition \ref{SL}, i.e., \[ f_{1}^{ss}(\mathcal{O}(\tilde{x}))=(k\subset k\times k\subset M_{2}(k)\subset Zorn(k)).\] Clearly, $f_{1}^{ss}$ is well defined. Now we proceed to prove the surjectivity of $f^{ss}_{1}$.

\vspace{2 mm}

 Let $(k\subset K\subset Q\subset C)$ be a maximal flag of composition algebras. Then the norm $N_{C}$ on $C$ is induced by a non-degenerate $K/k$-hermitian form with matrix representative of the form $\text{diag}(1,y_{1},y_{2},y_{3})$, where the norm of the quaternion subalgebra $Q$ is the trace form of the $K/k$-hermitian form $\text{diag}(1,y_{1})$, and $y_{1},y_{2},y_{3}\in k^{\times}$ is such that $y_{1}y_{2}y_{3}$ is a $K/k$-norm (see Section \ref{2}). First, we choose the orbit $\mathcal{O}(x)\in X^{ss}_{k}/G_{k}$ with \[x=-e_{1}e_{2}e_{3}-y_{0}e_{4}e_{5}e_{6}+y_{1}e_{1}e_{5}e_{6}+y_{2}e_{4}e_{2}e_{6}+y_{3}e_{4}e_{5}e_{3}\in X^{ss}_{k}.\] Then $f^{ss}(\mathcal{O}(x))=(K\subset C)$ (see Theorem \ref{KC}). Now, if we take $v=(a,0,0,b,0,0)^{t}\in V_{6}(k)$ with $a,b\in k^{\times}$, then for $\tilde{x}=(x,v)\in V_{k}$ we get from Proposition \ref{f2} \[f_{2}(\tilde{x})=-(y_{2}y_{3}a^{2}+y_{1}b^{2}+y_{0}ab).\] Clearly, there are many $a,b\in k^{\times}$ such that $f(\tilde{x})\neq 0$. Let us take $a=2$ and $b=-y_{0}/y_{1}$. Then, the stabilizer of the point $\tilde{x}$ with $f_{2}(\tilde{x})\neq 0$ is isomorphic to $SU(\text{diag}(y_{2},y_{3}))\simeq SL_{1}(Q)$ (see Section \ref{2}). To see this, we take the element \[g=\left(\begin{array}{cc}
\alpha & \beta \\
\gamma & \delta
\end{array}\right)\in Sp_{6}(k(\sqrt{-i})),\] where $\alpha,\beta,\gamma,\delta$ are the following matrices of order $3$.

    \[\alpha=\text{diag}(\frac{y_{0}+2\sqrt{-i}}{4\sqrt{-i}},1,1)\in M_{3}(k(\sqrt{-i})),\]
    \vspace{1 mm}
    \[\beta=\text{diag}(\frac{-2y_{1}}{y_{0}+2\sqrt{-i}}, \frac{-y_{2}}{2\sqrt{-i}}, \frac{-y_{3}}{2\sqrt{-i}})\in M_{3}(k(\sqrt{-i})),\]
    \vspace{1 mm}
    \[\gamma= \text{diag}(\frac{-(y_{0}^{2}+4i)}{8(\sqrt{-i})y_{1}}, \frac{-(y_{0}-2\sqrt{-i})}{2y_{2}},\frac{-(y_{0}-2\sqrt{-i})}{2y_{3}})\in M_{3}(k(\sqrt{-i})),\]
    \vspace{1 mm}
    \[\delta=\text{diag}(1,\frac{y_{0}+2\sqrt{-i}}{4\sqrt{-i}},\frac{y_{0}+2\sqrt{-i}}{4\sqrt{-i}})\in M_{3}(k(\sqrt{-i})).\] Then by direct computations we can check that \[g.(-e_{1}e_{2}e_{3}-2(\sqrt{-i})e_{4}e_{5}e_{6},(\frac{4\sqrt{-i}}{y_{0}+2\sqrt{-i}},0,0,\frac{-(y_{0}+2\sqrt{-i})}{2y_{1}},0,0)^{t})=(x,v)=\tilde{x},\] and the $k$-rational points in the stabilizer $gSL_{2}(k(\sqrt{-i}))g^{-1}\subset G(k(\sqrt{-i}))$ are the elements in $SU(\text{diag}(y_{2},y_{3}))$, where the embedding of $SL_{2}$ in $G=Sp_{6}$ is the same as described in Proposition \ref{SL}. Hence, we get \[f_{1}^{ss}(\mathcal{O}(\tilde{x}))=(k\subset K\subset Q\subset C),\] where $\mathcal{O}(\tilde{x})\in V^{ss}_{k}/G_{k}$. So, $f_{1}^{ss}$ is surjective.

\end{proof}

In the proof of the above theorem, the point $(x,v)\in V^{ss}_{k}$ has the stabilizer $SL_{1}(Q)\subset Sp_{6}(k)$, where \[x=-e_{1}e_{2}e_{3}-y_{0}e_{4}e_{5}e_{6}+y_{1}e_{1}e_{5}e_{6}+y_{2}e_{4}e_{2}e_{6}+y_{3}e_{4}e_{5}e_{3}\in X^{ss}_{k},\] $v=(2,0,0,\frac{-y_{0}}{y_{1}},0,0)^{t}\in V_{6}(k)$ and $SL_{1}(Q)\simeq SU(\text{diag}(y_{2},y_{3}))$. Now we consider the point $(x,v^{\prime})$ with $v^{\prime}=(0,2,0,0,\frac{-y_{0}}{y_{2}},0)^{t}\in V_{6}(k)$. For $(x,v^{\prime})$, $f_{1}$ has the same value as for $(x,v)$, and \[f_{2}((x,v^{\prime}))=-(4y_{1}y_{3}-\frac{y_{0}^{2}}{y_{2}})\neq 0.\] So, $(x,v^{\prime})\in V^{ss}_{k}$, and we can check that \[\text{Stab}((x,v^{\prime}))\simeq SU(\text{diag}(y_{1},y_{3})),\] similarly as we did in the above proof. In general, $SU(\text{diag}(y_{2},y_{3}))$ may not be isomorphic to $SU(\text{diag}(y_{1},y_{3}))$. In that case, the stabilizer of $(x,v^{\prime})\in V^{ss}_{k}$ is isomorphic to $SL_{1}(Q_{1})$, for some quaternion algebra $Q_{1}$ not isomorphic to $Q$. However, the octonion and quadratic algebras corresponding to the orbits of these two elements in $V^{ss}_{k}$ are the same (see Theorem \ref{KC}). Also note that $f_{2}((x,v))=-(4y_{2}y_{3}-y_{0}^{2}/y_{1})\neq f_{2}((x,v^{\prime}))$, if $Q_{1}$ and $Q$ are not isomorphic. This gives us an idea how the orbits are being changed when the quaternion in a maximal flag changes.

\vspace{2 mm}

\noindent \textbf{Canonical surjection:} The surjection in the proof of the above theorem, i.e., $f_{1}^{ss}$ appears in a very natural way. If $\tilde{x}=(x,v)\in V^{ss}_{k}\subset V=X\oplus V_{6}$ with $f_{1}(x)=i\in k^{\times}$, then $\text{Stab}(x)\simeq SU(h)\subset Sp_{6}(k)$, where $h$ is a trivial discriminant ternary $k(\sqrt{-i})/k$-hermitian form. It is a very well known fact that $SU(h)\simeq \text{Aut}(C/k(\sqrt{-i}))$, for some unique octonion $C$ containing $k(\sqrt{-i})$ as a subalgebra (see \cite{JN}, \cite{TML} for details). Again, $\text{Stab}(\tilde{x})\simeq SL_{1}(Q)\subset SU(h)\subset Sp_{6}(k)$ for some unique quaternion $Q$. From the above result, we get that \[f_{1}^{ss}(\mathcal{O}(\tilde{x}))=(k\subset k(\sqrt{-i})\subset Q\subset C).\] So, the maximal flag $f_{1}^{ss}(\mathcal{O}(\tilde{x}))$ corresponding to the $Sp_{6}(k)$-orbit of $\tilde{x}\in V^{ss}_{k}$ is being determined canonically by the stabilizers of the point $\tilde{x}$ and its $X$-component, inside the group $Sp_{6}(k)$.

\vspace{2 mm}

\begin{remark}
If $\tilde{x}=(x,v)\in V^{ss}_{k}$ is an orbit representative and $f_{1}^{ss}(\mathcal{O}(\tilde{x}))=(k\subset K\subset Q\subset C)$ for $\mathcal{O}(\tilde{x})\in V^{ss}_{k}/G_{k}$, then $f^{ss}(\mathcal{O}(x))=(K\subset C)$ for $\mathcal{O}(x)\in X^{ss}_{k}/G_{k}$.
\end{remark}

\vspace{2 mm}

\begin{remark}\label{fiber}
    Let $(k\subset K\subset Q\subset C)$ be a maximal flag. Then the fiber $(f_{1}^{ss})^{-1}\{(k\subset K\subset Q\subset C)\}$ is the following set \[\{\mathcal{O}((x,v))\in V^{ss}_{k}/G_{k}:f^{ss}(\mathcal{O}(x))=(K\subset C), \mathcal{O}(x)\in X^{ss}_{k}/G_{k},\text{Stab}((x,v))\simeq SL_{1}(Q)\}.\]
\end{remark}

\vspace{2 mm}

\begin{remark}
    Using the surjectivity of $f^{ss}_{1}$ we get $SL_{1}(Q)\subset Sp_{6}(k)$, for all quaternion algebras $Q$ defined over $k$. Also $SU(h)\subset Sp_{6}(k)$, for any trivial discriminant ternary $K/k$-hermitian form $h$ and $K/k$ any quadratic extension. We can also find out the embeddings of these groups inside $Sp_{6}(k)$, from the computations done earlier (see \cite{IJ} also).
\end{remark}

\vspace{2 mm}

\begin{remark}
    If we fix any composition algebra with dimension greater than $1$, the maximal flags containing that algebra are given by the following orbits.
    \begin{enumerate}
        \item Let us fix an octonion $C$. Then the orbits which correspond to the $4$-flags containing $C$ are \[\{\mathcal{O}((x,v))\in V^{ss}_{k}/G_{k}: f^{ss}(\mathcal{O}(x))=(K\subset C) \text{ for } \mathcal{O}(x)\in X^{ss}_{k}/G_{k}, \text{ dim}(K)=2\}.\]
        \item Let $Q$ be a fixed quaternion. The orbits which correspond to the $4$-flags containing $Q$ are \[\{ \mathcal{O}((x,v))\in V^{ss}_{k}/G_{k}:\text{Stab}((x,v))\simeq SL_{1}(Q)\}.\]
        \item Finally, we take any quadratic algebra $K=k(\sqrt{-i})$ for some $i\in k^{\times}$. Then the $4$-flags containing $K$ are represented by the orbits \[\{\mathcal{O}((x,v))\in V^{ss}_{k}/G_{k}: f_{1}(x)=ia^{2}, \text{ for some }a\in k^{\times}\}.\]
    \end{enumerate}
\end{remark}

\vspace{2 mm}

\subsection{Cohomological interpretation} Let $k\subset K\subset Q\subset C$ be a maximal flag of composition algebras over the field $k$. Then $H_{1}=\text{Aut}(C,K)\simeq SU(h)\rtimes \mathbb{Z}_{2}$ for some trivial discriminant $K/k$-hermitian form $h$ of rank $3$, and $H_{2}=\text{Aut}(C,Q)=SL_{1}(Q)\rtimes \text{Aut}(Q)$. Let us take $H=H_{1}\cap H_{2}$. So, $H$ is the group of all automorphisms of $C$ which map the subalgebras $Q$ and $K$ onto themselves and the cohomology set $H^{1}(k,H)$ is in bijection with the set of all isomorphism classes of maximal $4$-flags of composition algebras defined over $k$. This leads us to the following result.

\begin{theorem}
    We have a surjection from the orbit space $V^{ss}_{k}/G_{k}$ onto the cohomology set $H^{1}(k,H)$, i.e., each element in $H^{1}(k,H)$ is represented by some orbit in $V^{ss}_{k}/G_{k}$.
\end{theorem}

\begin{proof}
    The proof is immediate from Theorem \ref{CD}.
\end{proof}

\section{Some prehomogeneous vector spaces}\label{5} In this section, we discuss the orbit decompositions of the semi-stable sets for the prehomogeneous vector spaces $(Sp_{6}\times GL_{1}^{2}, V)$ and $(GSp_{6}\times GL_{1}^{2},V)$. In each case, we get that every orbit from the orbit space of the semi-stable set corresponds to a unique $4$-flag of composition algebras, as one can expect from the results discussed in Section \ref{4}.

\subsection{The PV $(\text{Sp}_{\text{6}}\times \text{GL}_{\text{1}}^{\text{2}},\text{V})$} We consider the $Sp_{6}$-representation $V=\wedge^{3}V_{6}=X\oplus V_{6}$ and define the action of $GL_{1}^{2}$ on $V$ by \[(a,b).(x,v)=(ax,bv), \text{ for } a,b\in GL_{1}, x\in X,v\in V_{6}.\] This makes $V$ a $(Sp_{6}\times GL_{1}^{2})$-representation, where the $Sp_{6}$-action is the same as defined earlier and we still have $V=X\oplus V_{6}$ as a $(Sp_{6}\times GL_{1}^{2})$-representation. We have the polynomials $f_{1}$ and $f_{2}$ defined on $V$, which are relative invariants with respect to the characters $\chi_{1}((g,a,b))=a^{4}$ and $\chi_{2}((g,a,b))=a^{2}b^{2}$, respectively, where $g\in Sp_{6}$ and $a,b\in GL_{1}$ (see Section \ref{3}).

\begin{proposition}\label{SpPV}
    The set $V^{ss}=\{\tilde{x}=(x,v)\in V=X\oplus V_{6}: f_{1}(x)\neq 0, f_{2}(\tilde{x})\neq 0 \}$ is a single $(Sp_{6}\times GL_{1}^{2})$-orbit in $V$ over $\overline{k}$. 
\end{proposition}

\begin{proof}
As we can see in Theorem \ref{fiber orbit}, any element $\tilde{x}=(x,v)\in V^{ss}$ with $f_{1}(x)=i\in k^{\times}$ can be assumed to be of the form \[\tilde{x}=(-e_{1}e_{2}e_{3}-2(\sqrt{-i})e_{4}e_{5}e_{6},(a,0,0,1,0,0)^{t}),\] up to $G$-equivalence over $k(\sqrt{-i})$, where $a$ is a non-zero scalar. If we take the element \[g=(\text{diag}((\sqrt{a})^{-1},\sqrt{a},1,\sqrt{a},(\sqrt{a})^{-1},1),1,1)\in (Sp_{6}\times GL_{1}^{2})(\overline{k}),\] then we have \[g.\tilde{x}= (-e_{1}e_{2}e_{3}-2(\sqrt{-i})e_{4}e_{5}e_{6},(\sqrt{a},0,0,\sqrt{a},0,0)^{t}).\] Now, we consider the element \[g_{1}=(I_{6},1,(\sqrt{a})^{-1})\in (Sp_{6}\times GL_{1}^{2})(\overline{k}),\] and we get \[g_{1}.(g.\tilde{x})=(-e_{1}e_{2}e_{3}-2(\sqrt{-i})e_{4}e_{5}e_{6},(1,0,0,1,0,0)^{t}).\] Finally, we take the element \[g_{2}=(\text{diag}(1,1,(-i)^{\frac{1}{4}},1,1,(-i)^{-\frac{1}{4}}), (-i)^{-\frac{1}{4}},1)\in (Sp_{6}\times GL_{1}^{2})(\overline{k}),\] and we have \[g_{2}.(g_{1}.(g.\tilde{x}))=(-e_{1}e_{2}e_{3}-2e_{4}e_{5}e_{6},(1,0,0,1,0,0)^{t}).\] So any element $\tilde{x}\in V^{ss}$ is equivalent to the above element under the $(Sp_{6}\times GL_{1}^{2})(\overline{k})$-action. Hence, $V^{ss}$ is a single orbit over $\overline{k}$.
\end{proof}

The above result tells us that $(Sp_{6}\times GL_{1}^{2},V)$ is a prehomogeneous vector space with $f_{1}f_{2}$ as relative invariant.

\vspace{2 mm}

\begin{remark}
    If we take the irreducible component $X\subset V$ and restrict the above action of $(Sp_{6}\times GL_{1}^{2})$, then it is enough to consider the action of $(Sp_{6}\times GL_{1})$ as the other $GL_{1}$ does not act on $X$. We can check that $(Sp_{6}\times GL_{1},X)$ is a PV with $f_{1}$ as a relative invariant (see \cite{SK}, \cite{PS}), and the orbit decompositions in this PV have been discussed in detail in (\cite{PS}).
\end{remark}

\vspace{2 mm}

\noindent \textbf{The identity components of the stabilizers:} Let $\tilde{x}=(x,v)$ be any point in the semi-stable set $V^{ss}_{k}$, i.e., $f_{1}(x)=i\neq 0$ and $f_{2}(\tilde{x})\neq 0$. As $Sp_{6}(k)\subset (Sp_{6}\times GL_{1}^{2})(k)$ and the stabilizer of $\tilde{x}\in V^{ss}_{k}$ in $Sp_{6}(k)$ is isomorphic to $SL_{1}(Q)$ for some quaternion algebra $Q$, we get that \[SL_{1}(Q)\subset \text{Stab}(\tilde{x})\subset (Sp_{6}\times GL_{1}^{2})(k).\] Again, $Sp_{6}\times GL_{1}^{2}$ has the Zariski-dense orbit $V^{ss}\subset V$ over $\overline{k}$. So, we must have \[\text{dim}(\text{Stab}(\tilde{x}))=\text{dim}(Sp_{6}\times GL_{1}^{2})-\text{dim}(V)=21+2-20=3.\] But $\text{dim}(SL_{1}(Q))=3$ as well. Hence, we get that the identity component of $\text{Stab}(\tilde{x})\subset (Sp_{6}\times GL_{1}^{2})(k)$ is $SL_{1}(Q)$. Similarly, we can check that the identity component of the stabilizer of $x\in X^{ss}_{k}$ inside $(Sp_{6}\times GL_{1}^{2})(k)$ is isomorphic to $SU(h)$, for some trivial discriminant $k(\sqrt{-i})/k$-hermitian form $h$ of rank $3$, using the fact that $(Sp_{6}\times GL_{1},X)$ is a PV.

\begin{theorem}\label{SpSS}

Each orbit in the orbit space $V^{ss}_{k}/(Sp_{6}\times GL_{1}^{2})(k)$ represents a unique isomorphism class of maximal flags of composition algebras over $k$, and all such flags up to isomorphism can be obtained in this way.
\end{theorem}

\begin{proof}

  If we can define a surjection $f^{ss}_{2}$ from the orbit space $V^{ss}_{k}/(Sp_{6}\times GL_{1}^{2})(k)$ onto the set of all isomorphism classes of maximal flags of composition algebras over $k$, then we are done.

    \vspace{2 mm}
    
    Let $\tilde{x}=(x,v)\in V^{ss}_{k}$ and $g=(a,b)\in GL_{1}^{2}(k)$. Then $g.\tilde{x}=(ax,bv)\in V^{ss}_{k}$ and $f_{1}(ax)=a^{4}f_{1}(x)$. We will show that the maximal flags corresponding to the $G_{k}$-orbits of $\tilde{x}$ and $g.\tilde{x}$ (determined by $f^{ss}_{1}$ in Theorem \ref{CD}) are the same. We have \[k(\sqrt{-f_{1}(x)})=k(\sqrt{-a^{4}f_{1}(x)}).\] So the quadratic algebras are the same in both the $4$-flags $f_{1}^{ss}(\mathcal{O}(\tilde{x}))$ and $f^{ss}_{1}(\mathcal{O}(g.\tilde{x}))$, where $\mathcal{O}(\tilde{x}),\mathcal{O}(g.\tilde{x})\in V^{ss}_{k}/G_{k}$. We may assume that \[x=-e_{1}e_{2}e_{3}-y_{0}e_{4}e_{5}e_{6}+y_{1}e_{1}e_{5}e_{6}+y_{2}e_{4}e_{2}e_{6}+y_{3}e_{4}e_{5}e_{3},\] for $y_{1},y_{2},y_{3}\in k^{\times}$ and $y_{1}y_{2}y_{3}=\frac{1}{4}y^{2}_{0}+f_{1}(x)$. Then the trivial discriminant $k(\sqrt{-f_{1}(x)})/k$-hermitian form $h$ given by the matrix $\text{diag}(y_{1},y_{2},y_{3})$ determines the octonion algebra in the $4$-flag $f^{ss}_{1}(\mathcal{O}(\tilde{x}))$ (see Theorem \ref{CD}, \ref{KC}). Now we have \[ax=-ae_{1}e_{2}e_{3}-ay_{0}e_{4}e_{5}e_{6}+ay_{1}e_{1}e_{5}e_{6}+ay_{2}e_{4}e_{2}e_{6}+ay_{3}e_{4}e_{5}e_{3}.\] Let us take the element $g_{1}=\text{diag}(a^{-1},1,1,a,1,1)\in G_{k}$, and we get \[g_{1}.(ax)=-e_{1}e_{2}e_{3}-a^{2}y_{0}e_{4}e_{5}e_{6}+y_{1}e_{1}e_{5}e_{6}+a^{2}y_{2}e_{4}e_{2}e_{6}+a^{2}y_{3}e_{4}e_{5}e_{3}.\] As we can see, the $k(\sqrt{-f_{1}(x)})/k$-hermitian form corresponding to the above element is given by the matrix $\text{diag}(y_{1},a^{2}y_{2},a^{2}y_{3})$, which is isometric $h$. Hence, the octonions are also the same in both $f^{ss}_{1}(\mathcal{O}(\tilde{x}))$ and $f^{ss}_{1}(\mathcal{O}(g.\tilde{x}))$. Finally, the connected components of the stabilizers of $\tilde{x}$ and $g.\tilde{x}$ are isomorphic over $k$ to $SL_{1}(Q)$ for some unique quaternion algebra $Q$, and so the quaternions in $f^{ss}_{1}(\mathcal{O}(\tilde{x}))$ and $f^{ss}_{1}(\mathcal{O}(g.\tilde{x}))$ must be the same (see Theorem \ref{CD}). Therefore, \[f^{ss}_{1}(\mathcal{O}(\tilde{x}))=f^{ss}_{1}(\mathcal{O}(g.\tilde{x})),\forall g\in GL_{1}^{2}(k),\] i.e., the $GL_{1}^{2}(k)$-action on $V^{ss}_{k}$ identifies two $ G_{k}$-orbits in $V^{ss}_{k}$ only if they are in the same fiber of the surjection $f_{1}^{ss}$. So, the required surjection $f^{ss}_{2}$ on $V^{ss}_{k}/(Sp_{6}\times GL_{1})^{2}(k)$ will be the map induced by $f^{ss}_{1}$, i.e., \[ f^{ss}_{2}(\mathcal{O}_{1}(\tilde{x}))=f^{ss}_{1}(\mathcal{O}(\tilde{x})), \text{ for }\tilde{x}\in V^{ss}_{k},\mathcal{O}_{1}(\tilde{x})\in V^{ss}_{k}/(Sp_{6}\times GL_{1})^{2}(k),\] where $\mathcal{O}(\tilde{x})\in V^{ss}_{k}/G_{k}$ is the $G_{k}$-orbit of $\tilde{x}$. Clearly, $f^{ss}_{2}$ is well defined. And it is surjective as $f^{ss}_{1}$ is so.
\end{proof}

\subsection{The PV ($\text{GSp}_{\text{6}}\times \text{GL}_{\text{1}}^{\text{2}}, \text{V}$)} Now we consider the general symplectic group $GSp_{6}$ which has been defined in Section \ref{3}. This group has a natural action on $V=\wedge^{3}V_{6}$, induced by its action on $V_{6}$. The decomposition of $V$ into irreducible components remains the same, i.e., $V=X\oplus V_{6}$ as $GSp_{6}$-representation. We define the action of $GL_{1}^{2}$ on $V$ in the same way as in the previous case. This makes $V$ into a $(GSp_{6}\times GL_{1}^{2})$-representation. It is very easy to see that $GSp_{6}\simeq Sp_{6}\rtimes H \simeq Sp_{6}\rtimes GL_{1}$ where \[GL_{1}\simeq H= \biggl\{h_{a}= \left(\begin{array}{cc}
aI_{3} & 0 \\
0 & I_{3}
\end{array}\right) : a\in GL_{1} \biggr\}\subset GSp_{6}.\] In addition, $f_{1}$ and $f_{2}$ are relative invariants for this representation, and the corresponding characters can be calculated from Proposition \ref{P2.1} (see Section \ref{3}). In particular, for the elements in $H$ the characters are defined by, respectively, \[\chi_{1}(h_{a})=a^{6}\text{ and } \chi_{2}(h_{a})=a^{4}, \text{ for } h_{a}\in H, a\in GL_{1}.\] This determines the characters $\chi_{1}$ and $\chi_{2}$ on the entire group $(GSp_{6}\times GL_{1}^{2})$, from the previous case.

\begin{proposition}
    The set $V^{ss}=\{\tilde{x}=(x,v)\in V: f_{1}(x)\neq 0, f_{2}(\tilde{x})\neq 0 \}$ is a single $(GSp_{6}\times GL_{1}^{2})$-orbit in $V$ over $\overline{k}$. 
\end{proposition}

\begin{proof}
    The proof follows from Proposition \ref{SpPV}, as $Sp_{6}\times GL_{1}^{2}\subset GSp_{6}\times GL_{1}^{2}$ and $f_{1},f_{2}$ are relative invariants under the action of $GSp_{6}\times GL_{1}^{2}$.
\end{proof}

So, we find that $(GSp_{6}\times GL_{1}^{2},V)$ is a PV with the relative invariant $f_{1}f_{2}$ defined on $V$.

\vspace{2 mm}

\begin{remark}
    In this case also, if we consider the irreducible component $X\subset V$, it is enough to consider the action of $(GSp_{6}\times GL_{1})$ on $X$; and similar to the earlier case, $(GSp_{6}\times GL_{1},X)$ is again an irreducible PV with $f_{1}$ as a relative invariant (see \cite{PS}, \cite{YA1}). The orbit decompositions of this PV have been discussed in detail in (\cite{PS}).
\end{remark}

\vspace{2 mm}

\noindent \textbf{The identity components of the stabilizers:} Let $\tilde{x}=(x,v)\in V^{ss}_{k}$. Then we can check by direct computations that the stabilizer of $\tilde{x}$ in $(GSp_{6}\times GL_{1}^{2})(k)$ contains the subgroup \[H_{1}= SL_{1}(Q)\times \{ (t.I_{6},t^{-3},t^{-1}): t\in GL_{1}(k) \},\] where $SL_{1}(Q)$ is the stabilizer of $\tilde{x}$ in $Sp_{6}(k)$ (see \cite{YA1}, Proposition $4.5$ also). Again, $(GSp_{6}\times GL_{1}^{2},V)$ is a PV and $V^{ss}$ is a Zariski-dense orbit in $V$ over $\overline{k}$. So we must have \[\text{dim}(\text{Stab}(\tilde{x}))=\text{dim}(GSp_{6}\times GL_{1}^{2})-\text{dim}(V)=24-20=4.\] Therefore, the identity component of $\text{Stab}(\tilde{x})\subset (GSp_{6}\times GL_{1}^{2})(k)$ is $H_{1}$ itself. Similarly, we can check that the identity component of the stabilizer of $x\in X^{ss}_{k}$ in $(GSp_{6}\times GL_{1}^{2})(k)$ is isomorphic to $SU(h)\times GL_{1}$, for some trivial discriminant $k(\sqrt{-f_{1}(x)})/k$-hermitian form $h$ of rank $3$, using the fact that $(GSp_{6}\times GL_{1},X)$ is a PV.

\begin{theorem}
Each point in the orbit space $V^{ss}_{k}/(GSp_{6}\times GL_{1}^{2})(k)$ corresponds to a unique isomorphism class of maximal flags of composition algebras over $k$, and all such maximal flags up to isomorphism can be obtained from these orbits.
\end{theorem}

\begin{proof}
    Let $\tilde{x}=(x,v)\in V^{ss}_{k}$. The proof is similar to the proof of Theorem \ref{SpSS}. We will show that the $G_{k}$-orbits of the elements $\tilde{x}$ and $h_{a}.\tilde{x}$ represent the same maximal flags determined by the map $f^{ss}_{1}$ (see Theorem \ref{CD}), for all $h_{a}\in H_{k}$ described above. To see this, we first assume that \[x=-e_{1}e_{2}e_{3}-y_{0}e_{4}e_{5}e_{6}+y_{1}e_{1}e_{5}e_{6}+y_{2}e_{4}e_{2}e_{6}+y_{3}e_{4}e_{5}e_{3},\] for $y_{0}\in k$ and $y_{1},y_{2},y_{3}\in k^{\times}$. Clearly, the quadratic algebras in the maximal flags corresponding to the $G_{k}$-orbits of $\tilde{x}$ and $h_{a}.\tilde{x}$ are the same as \[k(\sqrt{-f_{1}(x)})=k(\sqrt{-a^{6}f_{1}(x)})=k(\sqrt{-f_{1}(h_{a}.x)}),\] for all $h_{a}\in H_{k},a\in k^{\times}$. If we take $g=\text{diag}(a^{-3},1,1,a^{3},1,1)\in G_{k}$, then we have \[g.(h_{a}.x)=-e_{1}e_{2}e_{3}-a^{3}y_{0}e_{4}e_{5}e_{6}+a^{-2}y_{1}e_{1}e_{5}e_{6}+a^{4}y_{2}e_{4}e_{2}e_{6}+a^{4}y_{3}e_{4}e_{5}e_{3}\in X^{ss}_{k}.\] The trivial discriminant $k(\sqrt{-f_{1}(x)})/k$-hermitian form corresponding to the above element is given by $\text{diag}(a^{-2}y_{1},a^{4}y_{2},a^{4}y_{3})$, which is isometric to the same for $x\in X^{ss}_{k}$ (see Theorem \ref{KC}). Hence, the octonions corresponding to $\tilde{x}$ and $h_{a}.\tilde{x}$ are the same (see Theorem \ref{CD}). Finally, if $SL_{1}(Q)$ is the stabilizer of $\tilde{x}$ in $Sp_{6}(k)$, then the stabilizer of $h_{a}.\tilde{x}$ is also isomorphic to $SL_{1}(Q)$ under the action of $Sp_{6}(k)$ (observe that $h_{a}SL_{1}(Q)h_{a}^{-1}\subset Sp_{6}(k),\forall h_{a}\in H_{k}, a\in k^{\times}$). Therefore, the quaternions are also the same (see Theorem \ref{CD}).

    \vspace{2 mm}
    
    So, we find that two orbits in $V^{ss}_{k}/(Sp_{6}\times GL_{1}^{2})(k)$ are identified by the action of $H_{k}$, only if they are in the same fiber of $f^{ss}_{2}$ (see Theorem \ref{SpSS}). Therefore, we can define a surjection $f^{ss}_{3}$ from the orbit space $V^{ss}_{k}/(GSp_{6}\times GL_{1}^{2})(k)$ onto the set of all isomorphism classes of maximal flags of composition algebras using $f^{ss}_{2}$, as we have defined $f^{ss}_{2}$ using $f_{1}^{ss}$ in Theorem \ref{SpSS}. So we are done.
    
\end{proof}

\begin{remark}
    For both the PV's discussed in this section, the orbit spaces in the semi-stable sets represent all possible maximal flags of composition algebras over the field $k$. The fibers of $f_{2}^{ss}$ and $f_{3}^{ss}$ can also be computed very easily as we did for $f_{1}^{ss}$ (see Section \ref{4}, Remark \ref{fiber}).
\end{remark}

\section{The case of Freudenthal algebras}\label{6} In this section, we describe some of our observations about Freudenthal algebras following from the previous two sections. We have already defined these algebras in Section \ref{2}. We will assume that the characteristic of the underlying field $k$ is different from $2$ and $3$ for this section.   

\vspace{2 mm}

Let $\mathcal{A}=\mathcal{H}_{3}(C,\Gamma)$ be a reduced Freudenthal algebra for some $\Gamma=\text{diag}(\gamma_{1},\gamma_{2},\gamma_{3})\in GL_{3}(k)$ and a composition algebra $(C,N_{C})$ over $k$. Then $\mathcal{A}$ is uniquely determined up to isomorphism by the isometry class of the trace form $T_{\mathcal{A}}$ defined on $\mathcal{A}$. We may assume that $T_{\mathcal{A}}$ is of the form \[ T_{\mathcal{A}}\simeq \langle 1,1,1\rangle \perp b_{N_{C}}\otimes \langle \gamma_{3}^{-1}\gamma_{2},\gamma_{1}^{-1}\gamma_{3},\gamma_{2}^{-1}\gamma_{1} \rangle \simeq \langle 1,1,1\rangle \perp b_{N_{C}}\otimes \langle -a,-b,ab \rangle, \] for some $a,b\in k^{\times}$ (see \cite{KMRT}, Chapter IX). The isometry class of $T_{\mathcal{A}}$ uniquely determines the isometry classes of the Pfister forms $N_{C}$ and $N_{C}\otimes \langle\langle a,b\rangle \rangle$, and vice versa (see \cite{KMRT}, Corollary $37.16$). So we get the following results.

\begin{theorem}\label{F1}

Each orbit in $V^{ss}_{k}/G_{k}$ represents a unique isomorphism class of reduced Freudenthal algebras of dimension $6$ over $k$, and all such algebras up to isomorphism can be obtained from these orbits.

\end{theorem}

\begin{proof}
    Any algebra of this type is of the form $\mathcal{H}_{3}(k,\Gamma)$ for some $\Gamma=\text{diag}(\gamma_{1},\gamma_{2},\gamma_{3})\in GL_{3}(k)$. From the above discussion it is uniquely determined up to isomorphism by the isometry class of the $2$-fold Pfister form $\langle \langle a,b\rangle \rangle$, which is always a norm of some quaternion algebra over $k$. So, the reduced Freudenthal algebras of dimension $6$ are in bijection with the isomorphism classes of quaternion algebras over $k$.

    \vspace{1 mm}

    It is enough to define a surjection $g^{ss}$ from the orbit space $V^{ss}_{k}/G_{k}$ onto the set of all isomorphism classes of reduced Freudenthal algebras of dimension $6$ over the base field $k$, as we have done in the earlier cases.
    
    \vspace{1 mm}
    
    We have a unique quaternion algebra, say $(Q,N_{Q})$, in the maximal $4$-flag $f^{ss}_{1}(\mathcal{O}(\tilde{x}))$ for any $\tilde{x}\in V^{ss}_{k}$, i.e., $\mathcal{O}(\tilde{x})\in V^{ss}_{k}/G_{k}$ (see Theorem \ref{CD}). If $N_{Q}=\langle \langle c,d\rangle \rangle$, we define $g^{ss}(\mathcal{O}(\tilde{x}))$ to be the reduced Freudenthal algebra $\mathcal{A}$ with trace form \[T_{\mathcal{A}}=\langle 1,1,1\rangle \perp \langle 1\rangle \otimes \langle -c,-d,cd \rangle.\] Clearly, $g^{ss}$ is well defined; and it is surjective as $f_{1}^{ss}$ is so.
\end{proof}

Any Freudenthal algebra of dimension $6$ is reduced over the field $k$ (see \cite{GPR}, Proposition $39.17$, $46.1$ and Theorem $46.8$). So, the orbit space $V^{ss}_{k}/G_{k}$ in the above theorem covers all Freudenthal algebras of dimension $6$ over $k$.

\vspace{2 mm}

\begin{remark}
    Let $\mathcal{A}$ be a Freudenthal algebra of dimension $6$ determined by the $2$-fold Pfister form $\langle \langle a,b\rangle \rangle$, which is the norm of a unique quaternion $Q$. Then we have \[(g^{ss})^{-1}\{\mathcal{A}\}=\{\mathcal{O}(\tilde{x})\in V^{ss}_{k}/G_{k}:\text{Stab}(\tilde{x})\simeq SL_{1}(Q)\}.\]
\end{remark}

\vspace{2 mm}

\begin{theorem}\label{F2}
 Each point in the orbit space $V^{ss}_{k}/G_{k}$ represents a unique isomorphism class of reduced Freudenthal algebras of dimension $9$ over the field $k$, and all such algebras up to isomorphism can be obtained in this way.
\end{theorem}

\begin{proof}
    These algebras are uniquely determined up to isomorphism by the $2$-tuples \[(\langle \langle a\rangle \rangle, \langle \langle a\rangle \rangle \otimes \langle \langle b,c\rangle \rangle),\] where the first component is a $1$-fold Pfister form and the second component is a $3$-fold Pfister form containing the first component as a subform, for $a,b,c\in k^{\times}$. Now, using the surjection $f^{ss}$ in Theorem \ref{KC} we can get a surjection $g_{1}^{ss}$ from $V^{ss}_{k}/G_{k}$ onto the set of all isomorphism classes of reduced Freudenthal algebras of dimension $9$ over $k$. In particular, if $\tilde{x}=(x,v)\in V^{ss}_{k}$ and for $\mathcal{O}(x)\in X^{ss}_{k}/G_{k}$ we have \[f^{ss}(\mathcal{O}(x))=(K\subset C);\] then $g_{1}^{ss}(\mathcal{O}(\tilde{x}))$, for $\mathcal{O}(\tilde{x})\in V^{ss}_{k}/G_{k}$, is the reduced Freudenthal algebra $\mathcal{A}$ determined by the $2$-tuple $(N_{K},N_{C})$ consisting of the norms of $K$ and $C$, respectively. Here, $N_{C}$ contains $N_{K}$ as a subform as $K\subset C$ and \[N_{C}\simeq N_{K}\otimes \langle\langle b_{1},c_{1}\rangle \rangle\] for some $b_{1},c_{1}\in k^{\times}$ (see \cite{EL}). So, the trace form of $\mathcal{A}$ is \[T_{\mathcal{A}}\simeq \langle 1,1,1\rangle\perp b_{N_{K}}\otimes \langle -b_{1},-c_{1},b_{1}c_{1} \rangle,\] which uniquely determines $\mathcal{A}$ up to isomorphism. The surjectivity of $g_{1}^{ss}$ follows from the surjectivity of $f^{ss}$, as all $1$-fold and $3$-fold Pfister forms are the norms of uniquely determined quadratic and octonion algebras, respectively. This completes the proof.
\end{proof}

\vspace{2 mm}

\begin{remark}
    Let $\mathcal{A}$ be a reduced Freudenthal algebra of dimension $9$ over the field $k$, determined by the $2$-tuple $(N_{K}, N_{K}\otimes \langle \langle a,b\rangle \rangle)$ for some $a,b\in k^{\times}$ and $N_{K}$ is the norm of the quadratic algebra $K$. Also, let $C$ be the octonion algebra with norm $N_{C}=N_{K}\otimes \langle \langle a,b\rangle \rangle)$. Then $K\subset C$ and we have (see Theorem \ref{KC}) \[(g_{1}^{ss})^{-1}\{\mathcal{A}\}=\{\mathcal{O}((x,v))\in V^{ss}_{k}/G_{k} : \mathcal{O}(x)\in (f^{ss})^{-1}\{(K\subset C)\}\subset X^{ss}_{k}/G_{k}\}.\]
\end{remark}

\vspace{2 mm}

\begin{remark}
    We can get similar surjections from the orbit spaces in the semi-stable sets for the PV's $(Sp_{6}\times GL_{1}^{2},V)$ and $(GSp_{6}\times GL_{1}^{2},V)$, onto the same sets as in Theorem \ref{F1} and Theorem \ref{F2}. This follows easily from the computations in Section \ref{5}, and the two theorems discussed above. So the orbit spaces in the semi-stable sets in these PV's can also be interpreted as the reduced Freudenthal algebras of dimensions $6$ and $9$.
\end{remark}

\vspace{2 mm}

\begin{remark} Let us fix some $\Gamma=\text{diag}(\gamma_{1},\gamma_{2},\gamma_{3})\in GL_{3}(k)$ and $\Delta$ be the set of all isomorphism classes of maximal flags of composition algebras over the field $k$. Then there is a surjection \[g_{2}^{ss}: V^{ss}_{k}/G_{k}\rightarrow \{ (\mathcal{H}_{3}(k,\Gamma)\subset \mathcal{H}_{3}(K,\Gamma)\subset \mathcal{H}_{3}(Q,\Gamma)\subset \mathcal{H}_{3}(C,\Gamma)):(k\subset K\subset Q\subset C)\in \Delta \},\] (follows from Theorem \ref{CD}; see \cite{KMRT}, Chapter IX, Theorem $37.13$ also). So, we get a description of the $4$-flags of reduced Freudenthal algebras for some fixed $\Gamma$, in terms of the orbit space $V^{ss}_{k}/G_{k}$.  
\end{remark}

\end{document}